\documentclass[11pt]{amsart}
\usepackage[usenames]{color}
\usepackage{fullpage}
\usepackage{amscd}
\usepackage{amssymb, latexsym}
\usepackage[all,2cell,ps]{xy}
\usepackage{mathdots}
\usepackage{pdftricks}
\theoremstyle{plain}
\newtheorem{thm}{Theorem}[section]
\newtheorem{theorem}[thm]{Theorem}

\newtheorem{corollary}[thm]{Corollary}

\newtheorem{lemma}[thm]{Lemma}

\newtheorem{prop}[thm]{Proposition}
\newtheorem{proposition}[thm]{Proposition}
\newtheorem{ques}[thm]{Question}
\newtheorem{fact}[thm]{Fact}
\newtheorem{obser}[thm]{Observation}
\newtheorem{conjecture}[thm]{Conjecture}

\theoremstyle{definition}
\newtheorem{de}[thm]{Definition}

\newtheorem{rem}[thm]{Remark}
\newtheorem{remark}[thm]{Remark}
\newtheorem{example}[thm]{Example}
\newtheorem{convention}[thm]{Convention}
\newtheorem{algorithm}[thm]{Algorithm}

\newcommand\slR{\mathchoice%
{\scalebox{1.0}[0.2]{\strut\kern-0.1ex\huge\raise 1em\hbox to 0.1ex{$/$\hss}}\mathbf{R}}
{\scalebox{1.0}[0.2]{\strut\kern-0.1ex\huge\raise 1em\hbox to 0.1ex{$/$\hss}}\mathbf{R}}
{\scalebox{1.0}[0.2]{\strut\kern-0.1ex\Large\raise 1em\hbox to 0.1ex{$/$\hss}}\mathbf{R}}
{\scalebox{1.0}[0.2]{\strut\kern-0.1ex\large\raise 0.8em\hbox to 0.1ex{$/$\hss}}\mathbf{R}}
}

\newcommand{\Z}{\mathbb{Z}}

\newcommand{\cz}[1]{{\color{red}{#1}\color{black}{}}}
\newcommand{\nie}[1]{{\color{blue}{#1}\color{black}{}}}

\newcommand{\id}{\mathrm{id}}

\newcommand{\aut}[1]{\mathrm{Aut}(#1)}
\newcommand{\dis}[1]{\mathrm{Dis}(#1)}
\newcommand{\lmlt}[1]{\mathrm{LMlt}(#1)}

\newcommand{\ld}{\backslash}

\newcommand{\Ker}{\mathop{\mathrm{Ker}}}

\newcommand{\Aut}{\mathop{\mathrm{Aut}}}

\newcommand{\GD}[1]{\mathcal{GD}(#1)}
\newcommand{\D}[1]{\mathcal{D}(#1)}

\newcommand{\LR}[1]{\mathcal{LR}(#1)}

\newcommand{\eqrel}[2]{\mathrel{\stackrel{\scriptstyle\eqref{#1}}{#2}}}
\numberwithin{equation}{section}

\begin{document}

\title{Indecomposable non-degenerate 2-permutational  solutions of the Yang-Baxter equation}

\author{P\v remysl Jedli\v cka}
\author{Agata Pilitowska}

\address{(P.J.) Department of Mathematics and Physics, Faculty of Engineering, Czech University of Life Sciences, Kam\'yck\'a 129, 16521 Praha 6, Czech Republic}
\address{(A.P.) Faculty of Mathematics and Information Science, Warsaw University of Technology, Koszykowa 75, 00-662 Warsaw, Poland}

\email{(P.J.) jedlickap@tf.czu.cz}
\email{(A.P.) A.Pilitowska@mini.pw.edu.pl}


\date{\today}

\begin{abstract}
We present a complete characterization of all indecomposable non-degenerate, not necessarily involutive, solutions of the Yang-Baxter equation of multipermutation level~2.
We show that every such solution is a homomorphic image of a special, ``largest'' solution called \emph{the universal} one. On the other hand we prove that there is much simpler description.
At first, on the product of a group~$\Z_n^2$ and an abelian group~$G$,  we construct some family of indecomposable non-degenerate solutions of the Yang-Baxter equation of multipermutation level~2. Next, applying Rosenbaum's theorem of subgroups of a semidirect product and isolating a triple: a subgroup of $G$, a subgroup of $\Z_n^2$ and one group homomorphism, we obtain a~full description of each epimorphism which gives the desired solutions. Such a construction provides a tool how to find (and possibly enumerate) all indecomposable non-degenerate solutions of multipermutation level~$2$.  We also argue that the automorphism group of the discussed solutions is regular.
\end{abstract}

\maketitle
\section{Introduction}
The Yang-Baxter equation is a fundamental equation occurring in mathematical physics. It appears, for example, in integrable models in statistical mechanics, quantum field theory or Hopf algebras~(see e.g. \cite{Jimbo, K}). Searching for its solutions has been absorbing researchers for many years.

Let us recall that, for a vector space $V$, a {\em solution of the Yang--Baxter equation} is a linear mapping $r:V\otimes V\to V\otimes V$ 
 such that
\begin{align*}
(id\otimes r) (r\otimes id) (id\otimes r)=(r\otimes id) (id\otimes r) (r\otimes id).
\end{align*}

Description of all possible solutions seems to be extremely difficult and therefore
there were some simplifications introduced by Drinfeld in \cite{Dr90}.
Let  $X$ be a basis of the space $V$ and let $\sigma:X^2\to X$ and $\tau: X^2\to X$ be two mappings. We say that $(X,\sigma,\tau)$ is a {\em set-theoretic solution of the Yang--Baxter equation} if
the mapping 
$$x\otimes y \mapsto \sigma(x,y)\otimes \tau(x,y)$$ extends to a solution of the Yang--Baxter
equation. It means that $r\colon X^2\to X^2$, where $r=(\sigma,\tau)$,  
satisfies the \emph{braid relation}:
\begin{equation}\label{eq:braid}
(id\times r)(r\times id)(id\times r)=(r\times id)(id\times r)(r\times id).
\end{equation}

If $(X, \sigma,\tau)$ is a solution then directly by braid relation we obtain for $x,y,z\in X$:
\begin{align}
\sigma_x\sigma_y&=\sigma_{\sigma_x(y)}\sigma_{\tau_y(x)} \label{birack:1}\\
\tau_{\sigma_{\tau_y(x)}(z)}\sigma_x(y)&=\sigma_{\tau_{\sigma_y(z)}(x)}\tau_{z}(y) \label{birack:2}\\
\tau_x\tau_y&=\tau_{\tau_x(y)}\tau_{\sigma_y(x)} \label{birack:3}
\end{align}


A solution is called {\em non-degenerate} if the mappings $\sigma_x=\sigma(x,\_)$ and $\tau_y=\tau(\_\,,y)$ are bijections,
for all $x,y\in X$. We say that a solution is \emph{bijective} if the mapping $r$ is invertible. In particular, it is {\em involutive} if $r^2=\mathrm{id}_{X^2}$. A solution is
\emph{square free} if $r(x,x)=(x,x)$, for every $x\in X$.

All solutions $(X,\sigma,\tau)$, we study in this paper, are set-theoretic and non-degenerate, so we will call them simply \emph{solutions}. 
The set $X$ can be of arbitrary cardinality. 
\vskip 3mm
In the last years, the main interest of researchers lied in the investigation of involutive solutions. A special emphasis was taken onto two classess of solutions:  so-called {\em multipermutation} and \emph{indecomposable} ones. The former 
property resembles nilpotency and it is connected to the nilpotency of left braces (see e.g. \cite{GIC12, CJVV22}). The latter ones could be treated as ``bricks'' to construct solutions and knowledge of them  allows one to build more complex solutions. With the appearance of non-involutive solutions,  firstly studied by Soloviev \cite{Sol}, Lu, Yan and Zhu \cite{LYZ} and later by Guarnieri, Vendramin \cite{GV} and many others, Vendramin proposed in \cite[Problem 23]{V19} to study non-involutive multipermutation solutions.

In \cite[Section 3.2]{ESS} Etingof, Schedler and Soloviev introduced, for each involutive solution $(X,\sigma,\tau)$, the equivalence relation $\sim$ on the set $X$: for each $x,y\in X$
\[
x\sim y\quad \Leftrightarrow\quad \tau(\_\,,x)=\tau(\_\,,y).
\]
They showed that the quotient set $X/\mathord{\sim}$ can be again endowed
with a structure of a solution.

For the non-involutive case, one can define
the relation:
\begin{equation}
 x\approx y \quad \Leftrightarrow \quad  {\sigma_x=\sigma_y} \wedge {\tau_x=\tau_y}.
\end{equation}

Lebed and Vendramin showed in \cite{LV} that  the relation $\approx$ induces a solution on the quotient set $X^{\mathrel{\approx}}$ in injective solutions. 
In \cite{JPZ19} the authors together with Zamojska-Dzienio generalized the result for any solution $(X,\sigma,\tau)$.  A substantially shorter proof has recently appeared in \cite{CJKAV}. 
These results justifies the following definition.
\begin{de}\label{ret}
Let $(X,\sigma,\tau)$ be a solution. The quotient solution $\mathrm{Ret}(X,\sigma,\tau):=(X^{\approx},\sigma,\tau)$  with $\sigma_{x^{\approx}}(y^{\approx})=\sigma_x(y)^{\approx}$ and $\tau_{y^{\approx}}(x^{\approx})=\tau_y(x)^{\approx}$, for $x^{\approx},y^{\approx}\in X^{\approx}$  and $x\in x^{\approx},\; y\in y^{\approx}$, is called the \emph{retraction} solution of $(X,\sigma,\tau)$. One can define \emph{iterated retraction} in the following way: ${\rm Ret}^0(X,\sigma,\tau):=(X,\sigma,\tau)$ and
	${\rm Ret}^k(X,\sigma,\tau):={\rm Ret}({\rm Ret}^{k-1}(X,\sigma,\tau))$, for any natural number $k>1$.
If there exists an integer~$k$ such that $\mathrm{Ret}^k(X,\sigma,\tau)$ has one element only then we say that
$(X,\sigma,\tau)$ has {\em multipermutation level $k$}.
\end{de}

In particular, a solution~$(X,\sigma,\tau)$ is said to be a {\em permutation} solution (or {\em multipermutation solution of level}~$1$) if $|\mathrm{Ret}(X,\sigma,\tau)|=1$. 
A non-permutation solution~$(X,\sigma,\tau)$ is said to be a
{\em multipermutation solution of level}~$2$ (or briefly, $2$-\emph{permutational} solution), if $|\mathrm{Ret}(X,\sigma,\tau)/\mathord{\approx}|=1$.

Gateva-Ivanova studied involutive multipermutation  solutions in many papers. For example, in \cite{GIC12} together with Cameron they gave an equational characterization of square-free involutive solutions of multipermutation level at most $k$. Later on, she investigated in \cite{GI18} in more details  square-free involutive solutions of multipermutation level at most $2$. 

In the paper we pay special attention to non-degenerate, not necesserly involutive, 
solutions with a low level of multipermutation.
Note that the study of such solutions attracts, among others, researchers interested in the quantum spaces (see e.g. \cite{GIM11}).  
In \cite{JP23a} we proved that solutions of multipermutation level at most $2$ happen to fall into two special classes: so-called $2$-\emph{reductive} ones and strictly $2$-permutational. 
We gave a combinatorial
construction of $2$-reductive 
solutions and we also proved that solutions associated to a skew left braces are 2-reductive if and only if the skew left brace is nilpotent of class~$2$. In \cite{HJP24} we, together with J. Hora, showed that each solution of multipermutation level at most $2$ may be constructed from $2$-reductive one using some bijections.

Moreover, we focus also on indecomposable solutions. A solution $(X,\sigma,\tau)$ is \emph{indecomposable} if the permutation group  $\mathcal{G}(X)=\left\langle \sigma_x, \tau_x\colon x\in X\right\rangle$  acts transitively on $X$.
The investigation of indecomposable solutions was initiated  by Etingof et. al. in \cite{ESS}. They showed there that such solutions of prime cardinality are permutation solutions with cyclic permutation group. In \cite{JPZ22} the authors of this paper together with Zamojska-Dzienio gave a complete system of three invariants for finite non-isomorphic  indecomposable involutive solutions with cyclic permutation groups (cocyclic solutions). Smoktunowicz and Smoktunowicz presented in \cite{SS18} a construction of all finite indecomposable solutions based on one-generated left braces. Also Rump in \cite{Rump20} gave another method of constructing indecomposable solutions from left braces. In \cite{JP23} we give a complete characterization of all indecomposable involutive solutions of multipermutation level~2. Dietzel, Properzi, and Trappeniers 
gave in \cite{DPT25} a full classification of indecomposable involutive solutions which are of size $p^2$, for a prime $p$. Cedó and Okniński characterized in \cite{CO21} so called {\em simple} solutions (solutions with only trivial congruences). In particular, they proved that each finite simple solution (not of a prime order) is indecomposable. 
Colazzo, Jespers, Kubat, and Van Antwerpen presented in \cite{CJKV23} a characterisation of simple finite 
non-degenerate solutions in terms of the algebraic structure of the associated permutation skew left braces.   

In \cite{GIC12} Gateva-Ivanova and Cameron showed that square-free involutive solutions (of arbitrary cardinality)
with finite multipermutation level are decomposable. Additionally, Ced\'{o} and Okni\'{n}ski proved in \cite{CO21} that
indecomposable involutive solutions of square-free cardinality are always multipermutation solutions. Castelli, Mazzotta and Stefanelli presented an example \cite[Example 3.9]{CMS} of 3-elements indecomposable square-free solution which is not multipermutation. They also extended the result of Gateva-Ivanova and Cameron to non-involutive solutions and they proved in \cite{CMS} that finite non-degenerate multipermutation square-free solutions are decomposable. In \cite{JP24} we generalized their results and we  showed that each  non-degenerate square-free solution 
of multipermutation level $k$ and 
arbitrary cardinality was always decomposable. Just recently, Castelli discribed the precise relationship between indecomposable solutions and finite one-generator skew braces \cite{C25}.

Unfortunately, it is not easy to find, in general, all the indecomposable solutions. Hence we decided
to focus on a smaller class of indecomposable solutions --
 multipermutation solutions of level $2$.

In this paper we present a complete characterization of all such solutions.
We show that each of them is a homomorphic image of a particular solution defined on a product of a group $\Z_n^2$ and an abelian group $G$. Each epimorphism which gives the desired solutions is defined by a triple: a subgroup $H$ of $G$, a subgroup $S$ of $\Z_n^2$ and one group homomorphism $\Theta\colon S\to G/H$. Such construction provides a tool how to find (and possibly enumerate) all indecomposable solutions of multipermutation level~$2$. The paper is a significant extension of the results obtained in \cite{JP23} for involutive solutions. But contrary to the previous constructions, now our results are fully based on group theory, in particular on Rosenbaum's theorem of subgroups of a semidirect product.

The paper is organized as follows. 
In Section 2 we recall some known facts concerning solutions and we
present the definitions of important groups that act on a~solution with special emphasis to diagonal mappings of a solution. In Section 3 we present a~special kind of a notion used in quasigroup theory: isotopy. Section 4 is devoted to general properties of indecomposable  solutions of multipermutation level 2. In particular we show that the automorphism group of the discussed solutions is regular (Proposition \ref{prop:aut-regular}). Section \ref{sec:constr} contains a construction of a special class of indecomposable solutions of multipermutation level 2 (Theorem \ref{th:main2}) which is crucial for the paper. In Section 6 we introduce so called \emph{universal} and \emph{canonical} solutions. We 
study homomorphic images of the solutions we have constructed in the previous section (Propositions \ref{prop:hom-img} and \ref{prop:img-universal}). We prove that
every indecomposable solution of multipermutation level~$2$ can be obtained as such an image. In the last Section 7 we describe all epimoprhic images of canonical solutions constructed in Section 6. The main result, Theorem \ref{thm:import} gives the full characterization of all indecomposable solutions of multipermutation level~$2$. Finally, we present how to find all such solutions.

\vskip 5mm

\section{Preliminaries}

In this section we recall some known facts and we
present the definitions of some important permutation groups that act on a~solution.

\begin{de}
    Let $(X,\sigma,\tau)$ be a solution. We define
    \[
    U_X(x)=\sigma^{-1}_x(x),\qquad T_X(x)=\tau_{x}^{-1}(x),
    \]
    for all~$x\in X$. Both the mappings are commuting permutations~\cite{Rump19}, \cite{JP24}. If the solution
    is clear from the context, we drop the indices, writing $U(x)$ or $T(x)$ only.
\end{de}

Note that by \cite[Lemma 3.2]{JP24}, for any $x\in X$, 
\begin{align}
\sigma_{T^{-1}(x)}=\sigma_{U(x)}\;  {\rm and}\;  \tau_{U^{-1}(x)}=\tau_{T(x)}. \label{eq:LU}
\end{align}
This implies, for each $x\in X$
\begin{align}\label{eq:sigma_UT}
    \sigma_{UT(x)}=\sigma_x\quad {\rm and}\quad \tau_{UT(x)}=\tau_x.
\end{align}

Let us now discuss permutation groups that act on a non-degenerate solution.
Decades ago, when people were studying involutive solutions only, the notation $\mathcal{G}(X)$ was used for denoting the group
$\langle \sigma_x:x\in X\rangle$. For non-involutive solution the permutations $\sigma_x$ do not suffice for describing the structure of a solution and, as a consequence, two approaches emerged: one consists of taking a group of pairs, either of type $(\sigma_x,\hat\sigma_x)$ (\cite{B18}) 
or of type $(\sigma_x,\tau^{-1}_x)$ (\cite{CJKAV}). The other approach uses the following definition:

\begin{de}
The permutation group $\mathcal{G}(X)$ 
of the solution $(X,\sigma,\tau)$ is a subgroup of $S(X)$ defined in the following way:
$$\mathcal{G}(X)=\langle\sigma_x,\tau_x:x\in X\rangle.$$
\end{de}

We use this definition because such a group is easier to handle, although it is not a generalization of the group $\mathcal{G}(X)$, as defined for involutive groups, as we shall see on the following example:

\begin{example} [with Vicent Pérez-Calabuig]\strut\\
    Let us consider a set $X=\{0,1,2,3\}$ with the following operations:
    \begin{align*}
    \sigma_0&=(0,2) & \tau_0&=(0,1)\\
    \sigma_1&=(0,1,2,3) & \tau_1&=(0,3,1,2)\\
    \sigma_2&=(3,2,1,0) & \tau_2&=(0,2,1,3)\\
    \sigma_3&=(1,3) & \tau_3&=(2,3)
    \end{align*}
    Then $(X,\sigma,\tau)$ is an involutive solution where the group $\langle\sigma_x\rangle$ has only 8 elements, whereas the group $\langle\sigma_x,\tau_x\rangle$ is the symmetric group on~$X$. 
    The groups $\langle \sigma_x\rangle$
    and $\langle\tau_x\rangle$ are thus not identical but they are conjugated via the
    mapping $T=(0,1,3,2)$.
\end{example}

\begin{lemma}\label{lem:inv}
    A solution~$(X,\sigma,\tau)$ is involutive if and only if $U=T^{-1}$ and $(\sigma_x U)^{-1}=\tau_x T$, for all $x\in X$.
\end{lemma}

\begin{proof}
    The forward direction was proved in \cite[Proposition 2.1]{ESS}. Suppose hence $U=T^{-1}$ and $(\sigma_x U)^{-1}=\tau_x T$, for all $x\in X$.
    Then, for all $x,y\in X$,
    \[
        \sigma_{\tau_x(y)}^{-1}\tau_x(y)=U\tau_x(y) = \sigma_x^{-1}T^{-1}(y)=
        \sigma_x^{-1}U(y)
        =\sigma_x^{-1}\sigma_y^{-1}(y)
        \stackrel{\eqref{birack:1}}=
        \sigma^{-1}_{\tau_x(y)}\sigma^{-1}_{\sigma_y(x)}(y)
  \]
    and hence $\tau_x(y)=\sigma^{-1}_{\sigma_y(x)}(y)$, proving involutivity.
\end{proof}

Inspired by this result, we define other permutation groups acting on the set $X$:
\begin{de}
Let $(X,\sigma,\tau)$ be a solution. We define
\begin{align*}
    \D{X}&=\langle U,T\rangle,\\
    \GD{X}&=\langle\sigma_x,\tau_x,U,T:x\in X\rangle,\\
    \mathcal{LR}(X)&=\langle \sigma_xU,\tau_xT: x\in X\rangle,\\
    \dis{X}&=\langle \sigma_x\sigma_y^{-1},\tau_x\tau_y^{-1}:x,y\in X\rangle.
\end{align*}
\end{de}
The group Dis$(X)$ is called in literature \emph{displacement group of the solution} and it also plays important role in the theory of racks and quandles.
Clearly $\dis{x}\leq\LR{X}\leq\GD{X}$.

\begin{corollary}
    Let $(X,\sigma,\tau)$ be an involutive solution. Then $\mathcal{G}(X)=\langle\sigma_x,\tau_x:x\in X\rangle$ is a normal subgroup of~$\GD{X}$.
\end{corollary}

\begin{lemma}\label{lem:isom-LR-dis}
    Let $(X,\sigma,\tau)$ be a solution and let~$e\in X$. If $\LR{X}$ is abelian then
    $\LR{X}/\LR{X}_e\cong \dis{X}/\dis{X}_e$.    
\end{lemma}

\begin{proof}
    Observe that $\LR{X}=\langle \dis{X}\cup\{\sigma_eU,\tau_eT\}\rangle$,
    $\sigma_eU(e)=\sigma_e\sigma_e^{-1}(e)=e$
    and $\tau_eT(e)=e$. Hence $\LR{X}=\dis{X}\LR{X}_e$.
    The claim is then obtained from the 2nd 
    Isomorphism Theorem.
\end{proof}

Now we recall several notions from~\cite{JP23a} describing certain solutions of multipermutation level~$2$.

\begin{de}[\cite{JP23a}]
A solution $(X,\sigma, \tau)$ is $2$-\emph{reductive} if, for every $x,y\in X$:
\begin{align}
& \sigma_{\sigma_x(y)}=\sigma_y, \label{eq:red1}\\
& \tau_{\tau_x(y)}=\tau_y, \label{eq:red2}\\
& \sigma_{\tau_x(y)}=\sigma_y, \label{eq:red3}\\
& \tau_{\sigma_x(y)}=\tau_y. \label{eq:red4}
\end{align}
\end{de}

We can alternatively say that a solution is $2$-reductive if and only if its reduct is trivial~\cite[Corollary 3.21]{JP23a}.
The structure of $2$-reductive solution is rather combinatorial. Throughout all the paper, a $2$-reductive solution will be denoted by $(X,L,\mathbf{R})$ rather than $(X,\sigma,\tau)$.

\begin{prop}[\cite{JP23a}]
    Let $(X,L,\mathbf{R})$ be a $2$-reductive solution. Then the group $\mathcal{G}(X)$ is abelian.
\end{prop}

\begin{theorem}\label{thm:affmesh}
Let $I$ be a non-empty set, $(A_i)_{i\in I}$ be a family of abelian groups over $I$, $\bigcup_{i\in I} A_i$ be the disjoint union of the sets $A_i$, $c_{i,j},d_{i,j}\in A_j$, for $i,j\in I$, be some constants. Then $(\bigcup_{i\in I} A_i,L,\mathbf{R}_x)$, where  for $x\in A_i$, $y\in A_j$,
\begin{equation}\label{eq:7.8}
L_x(y)=y+c_{i,j}\quad {\rm and} \quad \mathbf{R}_x(y)= x+d_{i,j},
\end{equation}
is a $2$-reductive solution. The solution is square-free if and only if $c_{i,i}=d_{i,i}=0$, for all $i\in I$.
\end{theorem}
\noindent
If we assume that 
\begin{align}\label{a:mesh}
A_j=\left<\{c_{i,j},d_{i,j}\mid i\in I\}\right>,\; {\rm for \;every} \;j\in I,
\end{align}
 then the solution has orbits of the action of $\langle\{L_x,\mathbf{R}_x\mid x\in X\}\rangle$ equal to $A_j$, $j\in I$ and each orbit is a permutational solution. 

We will denote the solution satisfying \eqref{a:mesh} by $\mathcal A=((A_i)_{i\in I},\,(c_{i,j})_{i,j\in I},\,(d_{i,j})_{i,j\in I})$ and call it \emph{the disjoint union, over a set $I$, of abelian groups}.

\begin{theorem}\label{th:2red}
A solution $(X,\sigma,\tau)$ is $2$-reductive if and only if it is a disjoint union, over a set $I$, of abelian groups. The orbits of the action of $\mathcal{G}(X)$ 
coincide with the groups.
\end{theorem}


\begin{de}\cite{JP24}
A solution $(X,\sigma, \tau)$ is $2$-\emph{permutational} if, for every $x,y,z\in X$:
\begin{align}
& \sigma_{\sigma_x(z)}=\sigma_{\sigma_y(z)}, \label{eq:per1}\\
& \tau_{\tau_x(z)}=\tau_{\tau_y(z)}, \label{eq:per2}\\
& \sigma_{\tau_x(z)}=\sigma_{\tau_y(z)}, \label{eq:per3}\\
& \tau_{\sigma_x(z)}=\tau_{\sigma_y(z)}. \label{eq:per4}
\end{align}
A solution is $2$-permutational, if and only if its multipermutation level is $2$ or less~\cite{JP24}.
\end{de}


Let us recall that a mapping $\Phi\colon Y\to X$ is a \emph{homomorphism} of two solutions $(Y,\lambda,\rho)$ and $(X,\sigma,\tau)$ if, for each $y\in Y$,
\[
\Phi\lambda_y=\sigma_{\Phi(y)}\Phi\quad \mathrm{and}
\quad \Phi\rho_y=\tau_{\Phi(y)}\Phi.
\]
$\Phi$ is an automorphism, if it is additionally bijection.

\section{Isotopes of solutions}

In this section we present a~special kind of a notion used in quasigroup theory: isotopy. In full generality,
an isotopy consists of permuting rows, columns and symbols of the multiplication table of a quasigroup. In our case, we just permute rows of $\sigma$ and columns of $\tau$.


\begin{de}
Let $(X,\sigma,\tau)$ be a solution and $\pi_1$ and $\pi_2$ be two bijections on the set $X$ satisfying for $x,y,z\in X$
\begin{align}
\sigma_x\pi_1\sigma_y&=\sigma_{\sigma_x\pi_1(y)}\pi_1\sigma_{\tau_y\pi_2(x)},\label{gis1}\\
\tau_x\pi_2\tau_y&=\tau_{\tau_x\pi_2(y)}\pi_2\tau_{\sigma_y\pi_1(x)},\label{gis3}\\
\tau_{\sigma_{\tau_y\pi_2(x)}\pi_1(z)}\pi_2\sigma_x\pi_1(y)&=
\sigma_{\tau_{\sigma_y\pi_1(z)}\pi_2(x)}\pi_1\tau_z\pi_2(y).\label{gis2}
\end{align}
The solution $(X,\lambda,\rho)$, with bijections $\lambda_x=\sigma_x\pi_1$ and $\rho_x=\tau_x\pi_2$, is called $(\pi_1,\pi_2)$-\emph{isotope} of $(X,\sigma,\tau)$.
In such a case we will say that solutions $(X,\sigma,\tau)$ and $(X,\lambda,\rho)$ are \emph{isotopic}. 
\end{de} 

Note that for every $x,y\in X$,
\begin{align}\label{dis1}
\lambda_x\lambda_y^{-1}&=\lambda_x\pi_1\pi_1^{-1}\lambda_y^{-1}=\sigma_x\sigma_y^{-1}\qquad
\text{ and }\\
\label{dis2}
 \rho_x \rho_y^{-1}&= \rho_x\pi_2\pi_2^{-1}\rho_y^{-1}=\tau_x\tau_y^{-1}.\qquad
\end{align}
\begin{lemma}\label{lm:disiso}
All isotopes of a given solution have the same displacement groups.
\end{lemma}

In our article, we shall be often working with a
$2$-reductive solution $(X,L,\mathbf{R})$.
In this case Conditions \eqref{gis1}--\eqref{gis2} clearly reduce to the following ones:
\begin{align}
L_x\pi_1 L_y&=L_{\pi_1(y)}\pi_1L_{\pi_2(x)},\label{is1}\\
\mathbf{R}_x\pi_2 \mathbf{R}_y&=\mathbf{R}_{\pi_2(y)}\pi_2\mathbf{R}_{\pi_1(x)},\label{is3}\\
\mathbf{R}_{\pi_1(y)}\pi_2 L_x\pi_1&=L_{\pi_2(x)}\pi_1 \mathbf{R}_y\pi_2,\label{is2}
\end{align}
for $x,y\in X$.

\begin{theorem}\cite{HJP24}\label{th:iso2red}
Let $(X,L,\mathbf{R})$ be a $2$-reductive solution and $\pi_1$ and $\pi_2$ be two bijections on the set $X$ satisfying \eqref{is1}--\eqref{is2}. Then the $(\pi_1,\pi_2)$-isotope of $(X,L,\mathbf{R})$ is a $2$-permutational solution.
\end{theorem}

The following lemma directly follows by \cite[Lemma 4.6]{HJP24}.

\begin{lemma} \label{lem:simple}
 Suppose that $\pi_1$ and $\pi_2$ are two automorphisms of a $2$-reductive solution ~$(X,L,\mathbf{R})$. Then \eqref{is1}--\eqref{is2} are equivalent to
 \begin{align}
     L_x&=L_{\pi_1\pi_2(x)},\label{isop1}\\
     \mathbf{R}_x&=\mathbf{R}_{\pi_2\pi_1(x)},\label{isop3}\\
     \pi_1\pi_2&=\pi_2\pi_1.\label{isop2}
 \end{align}
\end{lemma}

\begin{proof}
 Let $x,y\in X$. For \eqref{is1} we have:
   \[
   L_x\pi_1L_y=L_{\pi_1(y)}\pi_1L_{\pi_2(x)}=\pi_1L_{\pi_2(x)}
L_y=L_{\pi_1\pi_2(x)}\pi_1L_y    \]    
and similarly for \eqref{is3}. \eqref{is2} is clear.
\end{proof}



Conditions of the preceding lemma resemble \eqref{eq:sigma_UT}, if $U$ and $T$
happen to be automorphisms. It is so in the case
of solutions of multipermutation level~$2$.

\begin{theorem}\cite{HJP24}
    Let $(X,\sigma,\tau)$ be a $2$-permutational solution. Then
    the permutations $U$ and $T$ are commuting automorphisms
    of $(X,\sigma,\tau)$ satisfying \eqref{gis1}--\eqref{gis2}. The $(U,T)$-isotope of $(X,\sigma,\tau)$ is a square-free $2$-reductive solution.
\end{theorem}

If we denote by $(X,L,\mathbf{R})$ the $(U,T)$-isotope of a $2$-permutational solution then
$L_x=\sigma_xU$ and $\mathbf{R}_x=\tau_xT$ and
we see that $\LR{X,\sigma,\tau}=\mathcal{G}(X,L,\mathbf{R})=\LR{X,L,\mathbf{R}}=\GD{X,L,\mathbf{R}}$.

\begin{corollary}\label{corol:LR-abelian}
    Let $(X,\sigma,\tau)$ be a $2$-permutational solution. Then $\LR{X}$ is abelian.
\end{corollary}

Moreover, since, for $2$-permutational solutions, $U$ and $T$ are automorphisms
of~$(X,\sigma,\tau)$, we have $T\sigma_x T^{-1}=\sigma_{T(x)}$ and also $T L_x T^{-1}=T \sigma_x UT^{-1}=T\sigma_x T^{-1}U=\sigma_{T(x)}U=L_{T(x)}$.
Similarly also for other conjugations of $\sigma_x$, $\tau_x$, $L_x$ and $\mathbf{R}_x$ by $U$ or~$T$.

\begin{lemma}    
Let $(X,\sigma,\tau)$ be a $2$-permutational solution. Then
    \[
    \dis{X}=\left\{ L_{x_1}^{r_1}\ldots L_{x_k}^{r_k}\mathbf{R}_{y_1}^{s_1}\ldots\mathbf{R}_{y_m}^{s_m} : \sum_{i=1}^k r_i=\sum_{j=1}^ms_i=0\right\}.
    \]
\end{lemma}

\begin{proof}
    One inclusion follows from the fact that the set defined above is a group containing all the generators of $\dis{X}$ since $\sigma_x\sigma_y^{-1}=L_xL_y^{-1}$
    and $\tau_x\tau_y^{-1}=\mathbf{R}_x\mathbf{R}_y^{-1}$. The other inclusion follows trivially
    from the abelianess of $\LR{X}$.
\end{proof}

\begin{proposition}
    Let $(X,\sigma,\tau)$ be a $2$-permutational solution. Then $\mathcal{G}(X)$, $\LR{X}$ and $\dis{X}$ are normal subgroups of~$\GD{X}$.
\end{proposition}

\begin{proof}
    Conjugation by~$U$ or~$T$ was already discussed before. Now
    \[
    \sigma_y^{-1} L_x\sigma_y=UU^{-1}\sigma_y^{-1}
    L_x\sigma_yUU^{-1}=UL_y^{-1}L_xL_yU^{-1}=UL_xU^{-1}=L_{U(x)}
    \]
    and similarly for other possible conjugations.
\end{proof}

\begin{proposition}\label{Pr:DQ}
Let $(X,\sigma,\tau)$ be a 
$2$-permutational solution.
Then the natural actions of the groups 
$\LR{X}$
and $\dis{X}$ on $X$ have the same orbits.
\end{proposition}

\begin{proof}
    Let $x,y\in X$ be two elements such that $y=\alpha(x)$ with $\alpha\in\langle\{L_x,\mathbf{R}_x\mid x\in X\}\rangle$. Then there are $r_1,\ldots,r_k,s_1,\ldots,s_m\in \mathbb{Z}$ and $x_1,\ldots,x_k,$ $y_1,\ldots,y_m\in X$ such that $\alpha=L_{x_1}^{r_1}\ldots L_{x_k}^{r_k}\mathbf{R}_{y_1}^{s_1}\ldots\mathbf{R}_{y_m}^{s_m}$.

    Let $r=r_1+\ldots+r_k$ and $s=s_1+\ldots+s_m$. Clearly, we have $\beta=L_y^{-r}\mathbf{R}_{y}^{-s}\alpha\in \dis X$ and 
    $\beta(x) = L_y^{-r}\mathbf{R}_{y}^{-s}\alpha(x)= L_y^{-r}\mathbf{R}_{y}^{-s}(y) = y$.
\end{proof}

\begin{lemma}\label{lem:alpha}
    Let $(X,\sigma,\tau)$ be a $2$-permutational solution and let $\alpha\in \LR{X}$. Then, for all $x\in X$,
    \begin{equation}
        \sigma_{\alpha(x)}=\sigma_x \qquad \text{and}\qquad \tau_{\alpha(x)}=\tau_x.
        \label{eq:2-red-dis}
    \end{equation}
\end{lemma}

\begin{proof}
    Since $\LR{X}$ is abelian, there are $r_1,\ldots,r_k,s_1,\ldots,s_m\in \mathbb{Z}$ and $x_1,\ldots,x_k,$ $y_1,\ldots,y_m\in X$ such that $\alpha=L_{x_1}^{r_1}\ldots L_{x_k}^{r_k}\mathbf{R}_{y_1}^{s_1}\ldots\mathbf{R}_{y_m}^{s_m}$.
    Then
    $$\sigma_{\alpha(x)}=L_{L_{x_1}^{r_1}\ldots L_{x_k}^{r_k}\mathbf{R}_{y_1}^{s_1}\ldots\mathbf{R}_{y_m}^{s_m}(x)}U^{-1}=L_xU^{-1}=\sigma_x$$
    since the solution $(X,L,\mathbf{R})$ is $2$-reductive. The other claim is analogous.
\end{proof}

\section{Indecomposable solutions}
From now on, let us fix the following notation: we have $(X,\sigma,\tau)$, a solution of multipermutation level~$2$, and we recall that the mappings $U(x)=\sigma_{x}^{-1}(x)$ and $T(x)=\tau_{x}^{-1}(x)$ are automorphisms of~$(X,\sigma,\tau)$.
Let us denote, like in the previous section,
$L_x=\sigma_xU$ and $\mathbf{R}_x=\tau_xT$,
for each $x\in X$.


\begin{lemma}\label{lm:relTU}
    For every 
    $x,y\in X$, we have 
    \begin{align*}
    \sigma_{x}^{-1}\sigma_y&=\sigma_{U(y)}\sigma_{U(x)}^{-1} &
    \tau_{x}^{-1}\tau_y&=\tau_{T(y)}\tau_{T(x)}^{-1}\\
\sigma_{x}\tau_y&=\tau_{T(y)}\sigma_{T(x)}
&
\sigma_{x}\tau_y^{-1}&=\tau_{T(y)}^{-1}\sigma_{U(x)} \\
    \sigma_x\sigma_y&=\sigma_{T(y)}\sigma_{U(x)} &
    \tau_x\tau_y&=\tau_{U(y)}\tau_{T(x)}
    \end{align*}
\end{lemma}

\begin{proof}
    Firstly, 
    \begin{align*}
        \sigma_x^{-1}\sigma_y=UL_x^{-1}L_yU^{-1}=
    UL_yL_{x}^{-1}U^{-1}=L_{U(y)}UU^{-1}L_{U(x)}^{-1}=\sigma_{U(y)}\sigma_{U(x)}^{-1},
    \end{align*}
    and analogously for $\tau$.\\   
    Secondly, 
    \begin{align*}
        &\sigma_{x}\tau_y=L_{x}U^{-1}\mathbf{R}_yT^{-1}=L_{x}\mathbf{R}_{U^{-1}(y)}U^{-1}T^{-1}
    =\mathbf{R}_{U^{-1}(y)}L_{T^{-1}T(x)}T^{-1}U^{-1}=\\
    &\mathbf{R}_{T(y)}T^{-1}L_{T(x)}U^{-1}
    =\tau_{T(y)}\sigma_{T(x)}.
    \end{align*}
    Thirdly, 
    \begin{align*}
    &\sigma_{x}\tau_y^{-1}=L_{x}U^{-1}T\mathbf{R}_y^{-1}=L_{x}\mathbf{R}_{U^{-1}T(y)}^{-1}U^{-1}T
    =\mathbf{R}_{TU^{-1}(y)}^{-1}L_{TT^{-1}(x)}TU^{-1}=\\
    &T\mathbf{R}^{-1}_{U^{-1}(y)}L_{T^{-1}(x)}U^{-1}
    =\tau_{T(y)}^{-1}\sigma_{U(x)}.
    \end{align*}
    Lastly, 
    \begin{align*}\sigma_x\sigma_y=L_x U^{-1}L_y U^{-1}=L_xL_{U^{-1}(y)}U^{-2}=L_{U^{-1}(y)}U^{-1}L_{U(x)}U^{-1}=\sigma_{T(y)}\sigma_{U(x)},
    \end{align*}
    and analogously for~$\tau$.
\end{proof}

Since $\GD{X}$ is generated by $\LR{X}$ and $\D{X}$
and since $\LR{X}$ is a normal subgroup of $\GD{X}$, we obtain:

\begin{lemma}\label{lem:UTLR}
For every  $\varphi\in\GD{X}$ there are $\alpha,\alpha'\in \LR{X}$ 
and $r,s\in \mathbb{Z}$ such that $\varphi=U^r T^s\alpha=\alpha' U^rT^s$.
\end{lemma}


Directly by Lemma \ref{lem:UTLR} and Proposition \ref{Pr:DQ} we obtain
\begin{lemma}\label{rem:dis}
    For any  $\varphi\in\GD{X}$ and any $x\in X$, there exist $\alpha\in \dis{X}$ and $r,s\in \mathbb{Z}$ 
 such that
\begin{align*}
&\varphi(x)=
U^rT^s\alpha(x).
\end{align*}
\end{lemma}


Let us now choose~$\tilde 0\in X$ arbitrarily.
The choice of the element $\tilde 0$ might
be of some relevance for decomposable solutions but does not matter at all for indecomposable solutions because, as we shall see in Proposition~\ref{prop:aut-regular}, the automorphism group of such solutions is transitive. Let us denote $\tilde i=T^i(\tilde 0)$, for each $i\in\Z$. The set $\{\tilde i \mid i \in\Z\}$ is a subset of X only but it turns out that it may be an important subset of~$X$. 

\begin{lemma}\label{lem:3.4}
    For all $x\in X$, $i,j,k,l\in\Z$,
    \begin{align}
       \sigma^k_{\tilde i}\sigma^l_{\tilde j} \label{eq:3.9}&=\sigma_{\widetilde{j+k}}^l\sigma^k_{\widetilde{i-l}}, &
        \sigma^k_{\tilde i}\tau^l_{\tilde j}&=\tau_{\widetilde{j+k}}^l\sigma^k_{\widetilde{i+l}}, \\
        \tau^k_{\tilde i}\sigma^l_{\tilde j}&=\sigma_{\widetilde{j-k}}^l\tau^k_{\widetilde{i-l}}, &
        \tau^k_{\tilde i}\tau^l_{\tilde j}&=\tau_{\widetilde{j-k}}^l\tau^k_{\widetilde{i+l}},\\
     \sigma_{\tilde i}&=\sigma_{\tilde 1}^i\sigma_{\tilde 0}^{1-i},&
        \tau_{\tilde i}&=\tau_{\tilde 0}^{1-i}\tau_{1}^{i}. \label{eq:3.10}
   \end{align}
\end{lemma}

\begin{proof}
    The first four identities are obtained inductively using Lemma \ref{lm:relTU}. 
    Now \eqref{eq:3.10} is trivially true for $i\in\{0,1\}$. If $i>1$ then, using an induction on~$i$,
    \[
    \sigma_{\tilde i}=\sigma_{\tilde i}\sigma_{\tilde 0}\sigma_{\tilde 0}^{-1}
  \stackrel{\eqref{eq:3.9}}=
  \sigma_{\tilde 1}\sigma_{\widetilde{i-1}}\sigma_{\tilde 0}^{-1}=\sigma_{\tilde 1}\sigma_{\tilde 1}^{i-1}\sigma_{\tilde 0}^{2-i}\sigma_{\tilde 0}^{-1}=\sigma_{\tilde 1}^i\sigma_{\tilde 0}^{1-i}.
    \]
    If $i<0$ then
    \[
    \sigma_{\tilde i}=\sigma_{\tilde i}\sigma_{\tilde 0}^{-1}\sigma_{\tilde 0}\stackrel{\eqref{eq:3.9}}=
    \sigma_{\tilde 1}^{-1}\sigma_{\widetilde{i+1}}\sigma_{\tilde 0}
    =\sigma_{\tilde 1}^{-1}\sigma_{\tilde 1}^{i+1}\sigma_{\tilde 0}^{-i}\sigma_{\tilde 0}=
    \sigma_{\tilde 1}^i\sigma_{\tilde 0}^{1-i}.
    \]
    The calculation for~$\tau$ is analogous.
\end{proof}

\begin{lemma}
    Let $(X,\sigma,\tau)$ be an indecomposable solution of multipermutation level at most~$2$ such that $U$ or $T$ are identity 
    permutations. Then the multipermutation level is actually~$1$.
\end{lemma}

\begin{proof}
Suppose $U=\mathrm{id}_X$. Let us take $x,y\in X$.
Because of indecomposability there exists $\varphi\in\mathcal{G}(X)$ such that $\varphi(x)=y$. According to Lemma~\ref{rem:dis},  $y=U^iT^j\alpha(x)=T^j\alpha(x)$, for some $\alpha\in\dis{X}$ and $i,j\in \Z$. Then
$\sigma_{y}=\sigma_{T^j\alpha(x)}\stackrel{\eqref{eq:sigma_UT}}=\sigma_{U^{-j}\alpha(x)}=\sigma_{\alpha(x)}\stackrel{\eqref{eq:2-red-dis}}=\sigma_x$. The same argument can be used for $\tau$. Proof for $T=\mathrm{id}_X$ is analogous.
\end{proof}

\begin{prop}\label{prop:gen}
The following holds for an indecomposable solution $(X,\sigma,\tau)$ of multipermutation level~$2$:
    \begin{enumerate}
        \item[(i)] The group $\mathcal{G}(X)$ is generated by $\{\sigma_{\tilde 0},\sigma_{\tilde 1},\tau_{\tilde 0},\tau_{\tilde 1}\}$.
        \item[(ii)] The group $\dis{X}$ is generated by $\{\sigma_{\tilde i}\sigma_{\tilde 0}^{-1} \mid i\in \Z\}\cup \{\tau_{\tilde i}\tau_{\tilde 0}^{-1} \mid i\in \Z\}$.
        \item[(iii)] The group $\LR{X}$ is generated by $\{\sigma_{\tilde i}\sigma_{\tilde 0}^{-1} \mid i\in \Z\}\cup \{\tau_{\tilde i}\tau_{\tilde 0}^{-1} \mid i\in \Z\} \cup \{\sigma_{\tilde 0}U,\tau_{\tilde 0}T\}$.
        \item[(iv)] $\mathcal{G}(X)=\langle \dis{X}\cup\{\sigma_{\tilde 0},\tau_{\tilde 0}\}\rangle$.
    \end{enumerate}
\end{prop}

\begin{proof}
    Let $x\in X$ be choosen arbitrarily. Since $(X,\sigma,\tau)$ is indecomposable, there exists
    $\varphi\in\mathcal{G}(X)$ such that $x=\varphi(\tilde 0)$. According to Lemma \ref{rem:dis}, there exist $r,s\in \Z$ and $\alpha\in\dis{X}$ such that $\varphi(\tilde 0)=U^rT^s\alpha(\tilde 0)$. Now, by Lemma \ref{lem:UTLR}
    \[\sigma_x=\sigma_{\varphi(\tilde 0)}=\sigma_{U^rT^s\alpha(\tilde 0)}
    \stackrel{\eqref{eq:sigma_UT}}=\sigma_{T^{s-r}(\tilde 0)}=\sigma_{\widetilde{s-r}}\stackrel{\eqref{eq:3.10}}=\sigma_{\tilde 1}^{s-r}\sigma_{\tilde 0}^{1-s+r}
    \]
    and analogously for ~$\tau_x$. 
    The rest then follows.
\end{proof}

\begin{corollary}
    An indecomposable solution of multipermutation level~$2$ is generated by a single element which can be chosen arbitrarily.
\end{corollary}

\begin{proof}
At first note that $\tilde 1=T(\tilde 0)=T\mathbf{R}^{-1}_{\tilde 0}({\tilde 0})=\tau^{-1}_{{\tilde 0}}({\tilde 0})$. Then, by Proposition \ref{prop:gen}, the group
$\mathcal{G}(X)$ is generated by $\sigma_{\tilde 0}$, $\sigma_{\tau^{-1}_{\tilde 0}(\tilde 0)}$, $\tau_{\tilde 0}$ and $\tau_{\tau^{-1}_{\tilde 0}(\tilde 0)}$. The choice of the element~$\tilde 0$ is arbitrary.    
\end{proof}

Further, using a direct computation, we obtain:
\begin{corollary}
\label{corol:semi-regular}
    Let $(Y, \lambda, \rho )$ and $(X, \sigma, \tau)$ be two indecomposable solutions of multipermutation level~$2$. Let $\Phi_1 : Y \to X$ and $\Phi_2 : Y \to X$ be two homomorphisms of the solutions. If there exists $e \in Y$ such that $\Phi_1 (e) = \Phi_2 (e)$ then $\Phi_1 = \Phi_2$.
\end{corollary}




\begin{lemma}\label{lm:center}
Let $(X,\sigma,\tau)$ be an indecomposable solution of multipermutation level~$2$ and let $\upsilon\in\D{X}$ be such that there exists
$\alpha\in\LR{X}$ that $\upsilon\alpha(\tilde 0)=\tilde 0$. Then $\upsilon\in Z(\GD{X})$.
\end{lemma}

\begin{proof}
Since $\upsilon$ commutes with $\D{X}$, we prove that it commutes with all the four generators of~$\mathcal{G}(X)$. By Lemma \ref{lem:UTLR} there is $\alpha'\in \LR{X}$ such that $\alpha'\upsilon=\upsilon\alpha$. Then by Lemma \ref{Pr:DQ} we obtain:
    \begin{align*}
        [\upsilon^{-1},\sigma_{\tilde 0}]&=\upsilon \sigma_{\tilde 0}^{-1}\upsilon^{-1} \sigma_{\tilde 0}
        = \sigma_{\upsilon(\tilde 0)}^{-1}\sigma_{\tilde 0}
        = \sigma_{\alpha'\upsilon(\tilde 0)}^{-1}\sigma_{\tilde 0}
        =\sigma_{\upsilon\alpha(\tilde 0)}^{-1}\sigma_{\tilde 0}=\sigma_{\tilde 0}^{-1}\sigma_{\tilde 0}=\mathrm{id}_X,\quad {\rm and}\\     
        [\upsilon^{-1},\sigma_{\tilde 1}]&=\upsilon \sigma_{\tilde 1}^{-1}\upsilon^{-1} \sigma_{\tilde 1}
        =T\upsilon \sigma_{\tilde 0}^{-1}\upsilon^{-1}T^{-1} \sigma_{\tilde 1}=
        T\sigma_{\upsilon(\tilde 0)}^{-1}T^{-1}\sigma_{\tilde 1}=T\sigma_{\alpha'\upsilon(\tilde 0)}^{-1}T^{-1}\sigma_{\tilde 1}=T\sigma_{\tilde 0}^{-1}T^{-1}\sigma_{\tilde 1}=\mathrm{id}_X,   
    \end{align*}
    and analogously for~$\tau$.
\end{proof}

\begin{proposition}\label{prop:aut-regular}
Let $(X,\sigma,\tau)$ be an indecomposable solution of multipermutation level~$2$. Then
the automorphism group of the solution $(X,\sigma,\tau)$ is regular.
\end{proposition}

\begin{proof}
    Let $x\in X$. By Lemma \ref{rem:dis} there are $j,j'\in \Z$ and $\beta\in \dis{X}$ such that $x=U^jT^{j'}\beta(\tilde 0)$. Let us define, for all $k,k'\in\Z$ and $\gamma\in\LR{X}$, the mapping $\Phi_{k,k',\gamma}\colon X\to X$ in the following way:
    \[
    \Phi_{k,k',\gamma}(x)=\Phi_{k,k',\gamma} (U^jT^{j'}\beta(\tilde 0))=U^{k}T^{k'}U^jT^{j'}\beta\gamma(\tilde 0).
    \]
    We will prove that this mapping is well-defined. At first note that if, for some $i,i'\in \Z$ and $\alpha\in \dis{X}$, $U^iT^{i'}\alpha(\tilde 0)=U^jT^{j'}\beta(\tilde{0})$ then there is $\beta'\in \dis{X}$ such that $T^{-j'}U^{-j}U^iT^{i'}\beta'\alpha(\tilde 0)=\tilde{0}$. Hence, by Lemma \ref{lm:center}, $T^{i'-j'}U^{i-j}\in Z(\GD{X})$. This implies 
    \begin{multline*}
    \Phi_{k,k',\gamma}(U^jT^{j'}\beta(\tilde{0}))=U^{k+j}T^{k'+j'}\beta\gamma(\tilde 0)=
    U^{k+j}T^{k'+j'}\beta\gamma\beta^{-1}T^{-j'}U^{-j}U^iT^{i'}\alpha(\tilde 0)\\
    =
    U^{k+j}T^{k'+j'}\gamma T^{i'-j'}U^{i-j}\alpha(\tilde 0)=U^{k+j}T^{k'+j'}T^{i'-j'}U^{i-j}\gamma\alpha(\tilde 0)=\Phi_{k,k',\gamma}(U^iT^{i'}\alpha(\tilde 0)).
    \end{multline*}
    Then we will show that the mapping $\Phi_{k,k',\gamma}$ is a homomorphism: 
    \begin{multline*}
        \sigma_{\Phi_{k,k',\gamma}(U^iT^{i'}\alpha(\tilde 0))}\Phi_{k,k',\gamma}(U^jT^{j'}\beta(\tilde 0)) = \sigma_{U^{k+i}T^{k'+i'}\alpha\gamma(\tilde 0)} (U^{k+j}T^{k'+j'}\beta\gamma(\tilde 0))\\
        =\sigma_{\widetilde{k'+i'-k-i}}U^{k+j}T^{k'+j'}\beta\gamma(\tilde 0)
        =U^{k}T^{k'}\sigma_{\widetilde{i'-i}}
        U^jT^{j'}\beta\gamma(\tilde 0)=
        \Phi_{k,k',\gamma}(\sigma_{U^iT^{i'}\alpha(\tilde 0)}U^jT^{j'}\beta(\tilde 0)).
    \end{multline*}
Clearly, $\Phi_{-k,-k',\gamma^{-1}}$ is the inverse automorphism to $\Phi_{k,k',\gamma}$ and since
$\Phi_{k,k',\gamma}(\tilde 0)=U^kT^{k'}\gamma(\tilde 0)$, the automorphism group is transitive. According to Corollary \ref {corol:semi-regular} it is regular.
\end{proof}

Carsten Dietzel pointed to us that Ferran Cedó and Jan Okniński constructed  
a simple (and thus indecomposable) solution of size~$p^2$ 
with the order of its automorphism group coprime to~$p$. Hence, in general, the automorphism group of an indecomposable solution need not be transitive.

\section{A construction of  indecomposable  solutions of multipermutation level 2}\label{sec:constr}

In this section we shall construct a special family of
indecomposable $2$-permutational solutions.
For the sake of compatibility with further sections, the solutions obtained by this construction will be denoted by $(Y,\lambda,\rho)$ and the associated mappings generating the group $\LR{Y}$ will be denoted by stroked letters, that means
\textit{Ł} and $\slR$.

Let us denote the group $(\Z,+,0)$ by $\Z_{\infty}$ which will be helpful for our further considerations. Additionally, in the case $n=\infty$ we assume that $a\equiv b\pmod n$ if and only if $a=b$.

\begin{theorem}\label{th:main2}
Let $n\in \mathbb{N}\cup \{\infty\}$ and 
$(G,+,0)$ be an abelian group.
Further, let 
${\bf c}=(c_i)_{i\in\Z_n}\in G^{\Z_n}$ and ${\bf d}=(d_i)_{i\in\Z_n}\in G^{\Z_n}$ be sequences  
such that $c_{0}=0$,  $d_{0}=0$ and 
the group $(G,+,0)$ is generated by the set ~$\{c_{i}, d_{i}\mid i\in\Z_n\}$.
\noindent
Then $(Y,\lambda,\rho)=(G\times \Z_n^2
,\lambda,\rho)$ with
 \begin{align*} 
 &\lambda_{(a,i,i')}((b,j,j'))=(b+c_{i-i'+j'-j-1},j+1,j')\quad {\rm and}\\
 &\rho_{(a,i,i')}((b,j,j'))=(b+d_{i-i'+j'-j+1},j,j'+1) 
 \end{align*}
is an indecomposable solution of multipermutation level $2$. 

\end{theorem}
\begin{proof}
Let us first define two blocks of constants: $c_{i,i',j,j'}=c_{i-i'+j'-j}$ and $d_{i,i',j,j'}=d_{i-i'+j'-j}$ for $i,j,i',j'\in\Z_n$. 
Note that, for each  $i',j,j'\in \Z_n$,  
$(G,+,0)= \langle c_{i,i',j,j'}, d_{i,i',j,j'}\mid i\in \Z_n\rangle$. With these
matrices we define,
using Theorem \ref{thm:affmesh}, a square-free $2$-reductive solution
~$(G\times \Z_n^2,\textit{Ł},\slR)$ with
 \[\textit{Ł}_{(a,i,i')}((b,j,j'))=(b+c_{i,i',j,j'},j,j')\quad {\rm and}\quad 
 \slR_{(a,i,i')}((b,j,j'))=(b+d_{i,i',j,j'},j,j'),
 \]
 for $(a,i,i'), (b,j,j')\in  G\times \Z_n^2$. 
 Note that $\textit{Ł}_{(a,i,i')}=\textit{Ł}_{(0,i,i')}$ and $\slR_{(a,i,i')}=\slR_{(0,i,i')}$, for every $a\in G$ and $i,i'\in \Z_n$,

 Now, let $\pi_1,\pi_2:G\times \Z_n^2 \to G\times \Z_n^2$ be $\pi_1((a,i,i'))=(a,i+1,i')$ and $\pi_2((a,i,i'))=(a,i,i'+1)$. Clearly, these two are commuting permutations of order~$n$ on the set $G\times \Z_n^2$. Moreover, they are automorphisms of $(G\times\Z_n^2,\textit{Ł},\slR)$:
 \[
 \pi_1\textit{Ł}_{(a,i,i')}((b,j,j'))=(b+c_{i,i',j,j'},j+1,j')=
 (b+c_{i+1,i',j+1,j'},j+1,j')=\textit{Ł}_{\pi_1((a,i,i'))}(\pi_1((b,j,j'))),
 \]
 for all $i,j,i',j'\in\Z_n$ and $a,b\in G$. Analogously for the other three combinations of $\pi_1$ and $\pi_2$ with $\textit{Ł}$ and $\slR$.

Now, using these automorphisms, we shall construct
an isotope of the solution $(G\times\Z_n^2,\textit{Ł},\slR)$.
According to Lemma~\ref{lem:simple}, we need to check
Conditions \eqref{isop1} and \eqref{isop3}. 
%
 Let us compute, for $(a,i,i'), (b,j,j')\in  G\times \Z_n^2$, 
\begin{align*}
    \textit{Ł}_{\pi_1\pi_2((a,i,i'))}((b,j,j'))&=\textit{Ł}_{(a,i+1,i'+1)}((b,j,j'))=\\
    (b+c_{i+1-i'-1+j'-j},j,j')&=(b+c_{i-i'+j'-j},j,j')
    =\textit{Ł}_{(a,i,i')}((b,j,j')),
\end{align*}
and analogously we have $\slR_{\pi_1\pi_2((a,i,i'))}((b,j,j'))=\slR_{(a,i,i')}((b,j,j'))$.

Since Condition~\eqref{isop2} is obvious, we obtain, by Theorem \ref{th:iso2red},
that the solution $(G\times \Z_n^2,\lambda,\rho)$ with 
\begin{align*}
&\lambda_{(a,i,i')}((b,j,j'))=\textit{Ł}_{(a,i,i')}\pi_1((b,j,j'))=
\textit{Ł}_{(a,i,i')}((b,j+1,j'))=\\
&(b+c_{i,i',j+1,j'},j+1,j')=(b+c_{i-i'+j'-j-1},j+1,j'),\quad {\rm and}\\
&\rho_{(a,i,i')}((b,j,j'))=\slR_{(a,i,i')}\pi_2((b,j,j'))=\slR_{(a,i,i')}((b,j,j'+1))=\\
&(b+d_{i,i',j,j'+1},j,j'+1)=(b+d_{i-i'+j'-j+1},j,j'+1),
\end{align*}
for $(a,i,i'), (b,j,j')\in 
G\times \Z_n^2$,
is a 2-permutational one. 

Note also that 
\begin{align*}
   & \lambda^{-1}_{(a,i,i')}((b,j,j'))=(b-c_{i,i',j,j'},j-1,j')=(b-c_{i-i'+j'-j},j-1,j'),\quad {\rm and}\\
   & \rho^{-1}_{(a,i,i')}((b,j,j'))=(b-d_{i,i',j,j'},j,j'-1)=(b-d_{i-i'+j'-j},j,j'-1).
\end{align*}
The solution is indecomposable if the permutation group
$\mathcal{G}(Y)$ 
is transitive. 
Let us have $a\in G$ arbitrary. We first  prove that
$(a +c_i+d_j,0,0)$ is in the same orbit as $(a,0,0)$, for any $i,j\in \Z_n$:
\begin{multline*}
    \lambda_{(0,1,0)}^{-1}\lambda_{(0,i+1,0)}\rho_{(0,0,1)}^{-1}\rho_{(0,j,1)}(a,0,0){=}
    \lambda_{(0,1,0)}^{-1}\lambda_{(0,i+1,0)}\rho_{(0,0,1)}^{-1}((a+d_j,0,1))=\\
    =\lambda_{(0,1,0)}^{-1}\lambda_{(0,i+1,0)}((a+d_j,0,0))
    =\lambda_{(0,1,0)}^{-1}((a+c_i+d_j,1,0))
    =(a+c_i+d_j,0,0).
\end{multline*}
By an induction argument then $(a,0,0)$ lies in the same orbit as $(0,0,0)$, for any $a\in G$. 
Now, for any $b\in G$ and $j,j'\in\Z_n$,
\begin{align*}
    \lambda_{(0,j+1,j')} ((b,j,j'))&=(b,j+1,j') &
    \rho_{(0,j,j'+1)} ((b,j,j'))&=(b,j,j'+1)\\
    \lambda_{(0,j,j')}^{-1} ((b,j,j'))&=(b,j-1,j') &
    \rho_{(0,j,j')}^{-1} ((b,j,j'))&=(b,j,j'-1)
\end{align*}
and therefore the permutation group $\mathcal{G}(Y)$  is transitive on~$Y=G\times \Z_n^2$.
%
 \end{proof}

 We will denote the solution described above by $\mathcal{S}(G\times \Z_n^2,{\bf c}, {\bf d})$, for $n\in \mathbb{N}\cup\{\infty\}$, an abelian group $(G,+,0)$ and ${\bf c}, {\bf d}\in G^{\Z_n}$ 
specified in Theorem \ref{th:main2}. 
Let us note that $\mathcal{S}(\Z_1\times \Z_n^2,{\bf 0}, {\bf 0})$ is a permutation solution. 
Let $(Y,\lambda,\rho)=\mathcal{S}(G\times \Z_n^2,{\bf c}, {\bf d})$ for some abelian group $(G,+,0)$ and ${\bf c}, {\bf d}\in G^{\Z_n}$. 
For any $(a,i,i'), (b,j,j'), (x,k,k')\in G\times \Z_n^2$ we have:
\begin{align*}
  & U_Y((a,i,i'))=\lambda^{-1}_{(a,i,i')}((a,i,i'))=(a-c_{i-i'+i'-i},i-1,i')=(a,i-1,i')=\pi^{-1}_1((a,i,i'))\quad  {\rm and}\\
  &T_Y((a,i,i'))=\rho^{-1}_{(a,i,i')}((a,i,i'))=(a-d_{i-i'+i'-i},i,i'-1)=(a,i,i'-1)=\pi^{-1}_2((a,i,i')).
   \end{align*}
Also note that 
$$\lambda_{(a,i,i')}\lambda_{(b,j,j')}^{-1} ((x,k,k'))=\lambda_{(a,i,i')}(x-c_{j-j'+k'-k},k-1,k')=(x+c_{i-i'+k'-k}-c_{j-j'+k'-k},k,k')$$
and analogously $\rho_{(a,i,i')}\rho_{(b,j,j')}^{-1} ((x,k,k'))=(x+d_{i-i'+k'-k}-d_{j-j'+k'-k},k,k')$. And also recall, from the proof of the theorem,
\begin{align*}
\textit{Ł}_{(a,i,i')}((b,j,j'))&=\lambda_{(a,i,i')}U_Y((b,j,j'))=(b+c_{i-i'-j+j'},j,j')\\
 \slR_{(a,i,i')}((b,j,j'))&=\rho_ {(a,i,i')}T_Y((b,j,j'))=(b+d_{i-i'-j+j'},j,j').
 \end{align*}


\begin{remark}\label{rem:cansol}
In a solution $(Y,\lambda,\rho)=\mathcal{S}(G\times \Z_n^2,{\bf c}, {\bf d})$, 
for any $i,j\in \Z_n$ and $a\in G$ we have:
\begin{itemize}
    \item $o(U_Y)=o(T_Y)=n$ and $U_Y^{i}T_Y^{j}=\id_Y$ if and only if $n\mid i$ and $n\mid j$.
    \item  $\mathcal{LR}(Y)\cap\mathcal{D}(Y)=\{\id\}$,
    \item $\GD{Y}=\D{Y}\ltimes  \LR{Y}$.
\end{itemize}
\end{remark}
\begin{lemma}\label{lm:GisoLR}
For a solution $(Y,\lambda,\rho)=\mathcal{S}(G\times \Z_n^2,{\bf c}, {\bf d})$, we have 
$$(G,+,0)\cong\LR{Y}/\LR{Y}_{(0,0,0)}\cong \dis{Y}/\dis{Y}_{(0,0,0)},$$
where both the isomorphisms send $c_i\in G$ onto the coset represented by $\lambda_{(0,i,0)}\lambda_{(0,0,0)}^{-1}$ and
$d_i\in G$ onto the coset represented by $\rho_{(0,0,i)}\rho_{(0,0,0)}^{-1}$.
\end{lemma}
\begin{proof}
Note that the group $\LR{Y}$ 
is equivalently generated by the set $\{\textit{Ł}_{(0,i,0)},\slR_{(0,0,i)}:i\in \Z_n\}$. Let $\alpha\in \LR{Y}$. Since the group is abelian, there are $k_0,\ldots,k_{n-1},l_0,\ldots,l_{n-1}\in \Z_n$ such that $\alpha=\prod_{i=0}^{n-1}\textit{Ł}_{(0,i,0)}^{k_i}\prod_{j=0}^{n-1}\slR_{(0,0,j)}^{l_j}$. Moreover, $\alpha((0,0,0))=(\sum_{i=0}^{n-1}k_ic_i+\sum_{j=0}^{n-1}l_jd_j,0,0)$.

    Let $\phi\colon\LR{Y}\to G$ be 
    such that 
    \begin{align*}
    &\phi(\alpha):=\sum_{i=0}^{n-1}k_ic_i+\sum_{j=0}^{n-1}l_jd_j. 
    \end{align*}
Clearly,   $\phi(\textit{Ł}_{(0,i,0)})=c_i$,  $\phi(\slR_{(0,0,i)})=d_i$ and $\phi$ is a homomorphism of groups $\LR{Y}$ and $(G,+,0)$. 
 Since the group $(G,+,0)$ is generated by the set $\{c_i, d_i:i\in \Z_n\}$, 
 $\phi$ is a surjection.
 Furthermore, 
 \begin{multline*}
 \Ker\phi=\{\alpha\in \LR{Y}:\phi(\alpha)=0\}=\{\alpha\in \LR{Y}:\sum_{i=0}^{n-1}k_ic_i+\sum_{j=0}^{n-1}l_jd_j=0\}=\\
 \{\alpha\in \LR{Y}:\alpha((0,0,0))=(0,0,0)\}=\LR{X}_{(0,0,0)}.
 \end{multline*}
 Let $\bar\phi$ be the associated isomorphism
 $\LR{Y}/\Ker{\phi}\to G$. Since $\textit{Ł}_{(0,0,0)}\in\Ker\phi$, we can see that
 \[\bar\phi^{-1}(c_i)=\textit{Ł}_{(0,i,0)}\LR{Y}_{(0,0,0)}=\textit{Ł}_{(0,i,0)}\textit{Ł}_{(0,0,0)}^{-1}\LR{Y}_{(0,0,0)}=
 \lambda_{(0,i,0)}\lambda_{(0,0,0)}^{-1}\LR{Y}_{(0,0,0)}
 \]
 and analogously for~$d_i$.

 Finally, the group $\dis{Y}$ is generated by the set $\{\lambda_{(0,i,0)}\lambda_{(0,0,0)}^{-1},\rho_{(0,0,i)}\rho_{(0,0,0)}^{-1}:i\in \Z_n\}$,
 and hence $\phi|_{\dis{Y}}$ is a surjective homomorphism as well. The rest then follows.
\end{proof}

\begin{lemma}\label{lem:stabilizer}
    In a solution $(Y,\lambda,\rho)=\mathcal{S}(G\times \Z_n^2,{\bf c}, {\bf d})$, 
    $\GD{Y}_{(0,0,0)}=\LR{Y}_{(0,0,0)}$.
\end{lemma}

\begin{proof}
    Let $\gamma\in\GD{Y}$ be such that $\gamma((0,0,0))=(0,0,0)$. According to Lemma~\ref{lem:UTLR}, there exists $\alpha\in\LR{Y}$ and $i,j\in\Z$ such that $\gamma=U^iT^j\alpha$. There exists $a\in G$ such that $\alpha((0,0,0))=(a,0,0)$. Then
    $\gamma((0,0,0))=U^iT^j((a,0,0))=(a,-i,-j)$.
    Hence $i\equiv n\equiv j\pmod n$ and hence $U^i=T^j=\mathrm{id}_Y$ and therefore $\gamma\in\LR{Y}$.
\end{proof}

\section{Universal and canonical solution}
In this section we 
study homomorphic images of the solutions we have constructed in the previous section. We prove that
every indecomposable solution of multipermutation level~$2$ can be obtained as such an image.

A huge advantage of epimorphisms is that they
induce well defined homomorphisms of groups.
Let $\Phi\colon Y\to X$ be a surjective homomorphism of solutions $(Y,\lambda,\rho)$ and $(X,\sigma,\tau)$. By \cite[Lemma 3.10]{JP23} the mapping $\Phi_{\mathcal{G}}\colon \mathcal{G}(Y)\to \mathcal{G}(X)$, such that $\Phi_{\mathcal{G}}(\lambda_y)=\sigma_{\Phi(y)}$, for each $y\in Y$, is well defined homomorphism of groups.

Let $a,y\in Y$ and $\textit{Ł}_y=\lambda_yU_Y$, $\slR_y=\rho_yT_Y$
and let $\Phi_{\mathcal{GD}}\colon \GD{Y}\to \GD{X}$ be the extension of $\Phi_{\mathcal{G}}$ given by:
\begin{align*}
&\Phi_{\mathcal{GD}}(U_Y)\Phi(y):=\Phi(U_Y(y)),\\
&\Phi_{\mathcal{GD}}(T_Y)\Phi(y):=\Phi(T_Y(y)),\\
&\Phi_{\mathcal{GD}}(\textit{Ł}_a)\Phi(y):=\Phi(\textit{Ł}_a(y)),\\
&\Phi_{\mathcal{GD}}(\slR_a)\Phi(y):=\Phi(\slR_a(y)).
\end{align*}
 Then the above definitions give:
\begin{align*}
&\Phi_{\mathcal{GD}}(U_Y)\Phi(y)=\Phi(U_Y(y))=
\Phi(\lambda^{-1}_y(y))=\sigma^{-1}_{\Phi(y)}\Phi(y)=U_X\Phi(y)\quad \Rightarrow\quad \Phi_{\mathcal{GD}}(U_Y)=U_X, \quad {\rm and}
\\
&\Phi_{\mathcal{GD}}(\textit{Ł}_a)\Phi(y)=\Phi(\textit{Ł}_a(y))=\Phi(\lambda_aU_Y(y))=\sigma_{\Phi(a)}\Phi(U_Y(y))=\sigma_{\Phi(a)}\Phi_{\mathcal{GD}}(U_Y)\Phi(y)=\\
&\sigma_{\Phi(a)}U_X\Phi(y)=
L_{\Phi(a)}\Phi(y)\quad \Rightarrow\quad \Phi_{\mathcal{GD}}(\textit{Ł}_a)=L_{\Phi(a)}.
\end{align*}
Similarly,
\[
\Phi_{\mathcal{GD}}(T_Y)=T_X \qquad {\rm and}\qquad 
\Phi_{\mathcal{GD}}(\slR_a)=\mathbf{R}_{\Phi(a)}.
\]
Hence, for each $\chi\in \GD{Y}$, $\Phi_{\mathcal{GD}}(\chi)\Phi(y)=\chi'\Phi(y)$, for some $\chi'\in \GD{X}$.
\\
Further, let $\Phi_{\mathrm{Dis}}$ be the restriction
of~$\Phi_{\mathcal{GD}}$ to the group $\dis{Y}$. Clearly, $\Phi_{\mathrm{Dis}}\colon \dis{Y}\to \dis{X}$.
Let $a\in Y$. Then 
\begin{align*}
\Phi_{\mathrm{Dis}}(\dis{Y}_a)=\Phi_{\mathrm{Dis}}(\{\alpha\in \dis{Y}:\alpha(a)=a\})=\dis{X}_{\Phi(a)}
\end{align*}
and $\Phi_{\mathrm{Dis}}$ induces the homomorphism $\overline{\Phi}_{\mathrm{Dis}}\colon \dis{Y}/\dis{Y}_{a}\to \dis{X}/\dis{X}_{\Phi(a)}$
of quotient groups given by $\alpha\dis{Y}_{a}\mapsto \Phi_{\mathrm{Dis}}(\alpha)\dis{X}_{\Phi(a)}$.

\begin{prop}\label{prop:hom-img}
   Let $(X,\sigma,\tau)$ be a non-degenerate solution, let $(G,+,0)$ be an abelian group, $n\in\mathbb{N}\cup\{\infty\}$
   and let us have two series $\{c_i\}_{i\in\Z_n}$, and $\{d_i\}_{i\in\Z_n}$
   satisfying the conditions of Theorem~\ref{th:main2}.
   Then $(X,\sigma,\tau)$
   is a homomorphic image of $(Y,\lambda,\rho)=\mathcal{S}(G\times \Z_n^2,\mathbf{c},\mathbf{d})$ if and only if the following conditions hold:
    \begin{enumerate}
        \item $(X,\sigma,\tau)$ is indecomposable,
        \item multipermutation level of~$(X,\sigma,\tau)$ is at most~$2$,
        \item the order $o(U_X)$ of the permutation~$U_X$ divides~$n$,
        \item the order $o(T_X)$ of the permutation~$T_X$ divides~$n$,
        \item there exists a group homomorphism~$\varphi$ from~$(G,+,0)$ onto~$\dis{X}/\dis{X}_{e}$ such that, for all~$i\in\Z_n$, $\varphi(c_i)=\sigma_{U_X^{-i}(e)}\sigma_e^{-1}\dis{X}_e$ and  $\varphi(d_i)=\tau_{T_X^{-i}(e)}\tau_e^{-1}\dis{X}_e$, for some $e\in X$.
    \end{enumerate}
\end{prop}
\begin{proof}
\noindent ``$\Rightarrow$'': Let $(X,\sigma,\tau)$ be a homomorphic image of an epimorphism $\Phi\colon Y\to X$ of solutions and let $e:=\Phi((0,0,0))$. 
Homomorphic images of indecomposable solutions
of multipermutation level~$2$ are necessarily
indecomposable solutions of level at most~$2$. 
Since $\Phi_{\mathcal{GD}}(U_Y)=U_X$ and
$\Phi_{\mathcal{GD}}(T_Y)=T_X$, we see that $U_X^n=T_X^n=\mathrm{id}_X$. 


Let $\psi:G\to\dis{Y}/\dis{Y}_{(0,0,0)}$ be the isomorphism defined in Lemma~\ref{lm:GisoLR} and let $\overline\Phi_{\mathrm{Dis}}$ be the induced homomorphism $\dis{Y}/\dis{Y}_{(0,0,0)}\to \dis{X}/\dis{X}_{e}$. Then $\varphi=\overline\Phi_{\mathrm{Dis}}\circ\psi$
satisfies
\begin{align*}
    \varphi(c_i)&=
    \overline\Phi_{\mathrm{Dis}}(\lambda_{U_Y^{-i}(0,0,0)}\lambda_{(0,0,0)}^{-1}\dis{Y}_{(0,0,0)})
    =\sigma_{U_X^{-i}(e)}\sigma_e^{-1}\dis{X}_e,\\
    \varphi(d_i)&=
    \overline\Phi_{\mathrm{Dis}}(\rho_{T_Y^{-i}(0,0,0)}\rho_{(0,0,0)}^{-1}\dis{Y}_{(0,0,0)})
    =\tau_{T_X^{-i}(e)}\tau_e^{-1}\dis{X}_e.
\end{align*}

``$\Leftarrow$'':
Let $(X,\sigma,\tau)$ be a solution satisfying the conditions (1)--(5).
Let $(Y,\lambda,\rho):=\mathcal{S}(G\times \Z_n^2,\mathbf{c},\mathbf{d})$. 
Let us define $\Phi:Y\to X$ as
\[\Phi((a,i,i')):=U_X^{-i}T_X^{-i'}(\varphi(a)(e)),     
 \]
 for each $a\in G$ and $i,j\in\Z$,
 where $\varphi(a)(e)$  stands for the image of~$e$ under a permutation $\alpha\in\dis{X}$ such that $\varphi(a)=\alpha\dis{X}_e$. 
It is evident that the mapping $\Phi$ is well defined since $\varphi(a)(e)$ does not depend on the representative of the coset~$\varphi(a)$.
According to Lemma \ref{lem:UTLR} and Proposition~\ref{Pr:DQ}
this mapping is onto. We will prove that $\Phi$ is a homomorphism of solutions. Let $i,i',j,j'\in \Z_n$ and $a,b\in G$. Hence
\begin{multline*}
    \Phi(\lambda_{(a,i,i')}((b,j,j')))=
    \Phi((b+ c_{i-i'+j'-j-1},j+1,j'))=
   U_X^{-j-1}T_X^{-j'}(\varphi(b) \sigma_{U_X^{-i+i'-j'+j+1}(e)}\sigma^{-1}_{e}(e))\\
   \stackrel{\eqref{dis1}}=U_X^{-j-1}T_X^{-j'}(\varphi(b) L_{U_X^{-i+i'-j'+j+1}(e)}L^{-1}_{e}(e))\stackrel{\eqref{eq:LU}}=
    L_{U_X^{-i+i'}(e)} U_X^{-j-1}T_X^{-j'} (\varphi(b) (e))=\\
        \sigma_{U_X^{-i+i'}(e)} U_X^{-j}T^{-j'} (\varphi(b) 
        (e))
\stackrel{\eqref{eq:2-red-dis}}=
\sigma_{U_X^{-i}T_X^{-i'}(\varphi(a)(e))} U_X^{-j}T_X^{-j'}(\varphi(b)(e))=
\sigma_{\Phi((a,i,i'))}\Phi((b,j,j')).
\end{multline*}
Analogously, we can show that $\Phi(\rho_{(a,i,i')}((b,j,j')))=\tau_{\Phi((a,i,i'))}\Phi((b,j,j'))$, which finishes the proof.
\end{proof}
The previous result naturally suggests the following structure.
Let $(X,\sigma,\tau)$ be a finite indecomposable solution of level at most~$2$ and let $\tilde 0\in X$.
Let us take:
\begin{itemize}
    \item $(G,+,0)=\dis{X}/\dis{X}_{\tilde 0}\cong\LR{X}/\LR{X}_{\tilde 0}$,
    \item $n=\mathrm{lcm}(o(U),o(T))$,
    \item ${\bf c}=(c_i)_{i\in\Z_n}\in G^{n}$ be a sequence  
such that $c_i=\sigma_{\tilde i}\sigma^{-1}_{\tilde 0}\dis{X}_{\tilde 0}$
    \item ${\bf d}=(d_i)_{i\in\Z_n}\in G^{_n}$ be a sequence 
such that $d_i=\tau_{\tilde i}\tau^{-1}_{\tilde 0}\dis{X}_{\tilde 0}$.
\end{itemize}
In this case, $c_0=\sigma_{\tilde 0}\sigma^{-1}_{\tilde 0}\dis{X}_{\tilde 0}=\dis{X}_{\tilde 0}$ and $d_0=\tau_{\tilde 0}\tau^{-1}_{\tilde 0}\dis{X}_{\tilde 0}=\dis{X}_{\tilde 0}$. And, according to Proposition~\ref{prop:gen}, we have $G=\langle c_i,d_i:i\in \Z_n\rangle$.
\begin{de}
The solution $\mathcal{C}(X,\sigma,\tau):=\mathcal{S}(\dis{X}/\dis{X}_{\tilde 0}\times \Z_n^2,{\bf c}, {\bf d})$, with $n$, ${\bf c}$ and ${\bf d}$ defined as above, is called \emph{the canonical solution} of a solution $(X,\sigma,\tau)$.
\end{de}
\begin{corollary}\label{prop:canhom}
   Let $(X,\sigma,\tau)$ be an indecomposable solution of level at most~$2$. Then it is a homomorphic image of its canonical solution. 
\end{corollary}

As one might expect, the canonical solution
is the smallest construction that maps onto the given solution.

\begin{proposition}
    Let $(X,\sigma.\tau)$ be an indecomposable solution of multipermutation level~$2$ and let
    $(X,\sigma,\tau)$ be a homomorphic image of a solution  $\mathcal{S}(G\times \Z_n^2,\mathbf{c},\mathbf{d})$, for some abelian group~$(G,+,0)$, an integer~$n$ and two sequences~$\mathbf{c}$ and $\mathbf{d}$, along with a homomorphism~$\Psi$.
   Further, let $\Phi$ be the natural projection from the canonical solution $\mathcal{C}(X,\sigma,\tau)$ onto $(X,\sigma,\tau)$. 
   Then there exists a surjective homomorphism $\psi\colon \mathcal{S}(G\times \Z_n^2,\mathbf{c},\mathbf{d})\to \mathcal{C}(X,\sigma,\tau)$ 
   such that $\Psi=\Phi\circ\psi$.
\end{proposition}
\begin{proof}
    Let~$(Z,\mu,\nu)=\mathcal{C}(X,\sigma,\tau)$ be the canonical solution of $(X,\sigma,\tau)$ and let $(Y,\lambda,\rho)=\mathcal{S}(G\times \Z_n^2,\mathbf{c},\mathbf{d})$.    
According to Proposition~\ref{prop:hom-img},
there exists a epimorphism~$\varphi$ from $G$ onto
$\dis{X}/\dis{X}_{\tilde 0}$ such that
$\varphi(c_i)=\sigma_{\tilde i}\sigma^{-1}_{\tilde 0}\dis{X}_{\tilde 0}$ and $\varphi(d_i)=\tau_{\tilde i}\tau^{-1}_{\tilde 0}\dis{X}_{\tilde 0}$.
Now let
$\psi((a,i,i')):=(\varphi(a),i,i')$.
This mapping is well defined since $n$ is a multiple of $\mathrm{lcm}(o(U_X),o(T_X))$. It is clearly onto and
\begin{multline*}
    \psi(\lambda_{(a,i,i')}((b,j,j')))=
    \psi((b+c_{i-i'+j'-j-1},j+1,j'))=
    (\varphi(b)\sigma_{\widetilde{i-i'+j'-j-1}}\sigma_{\tilde 0}^{-1},j+1,j')\\
    =\mu_{(\varphi(a),i,i')}((\varphi(b),j,j')).
\end{multline*}
Analogously, we get $\psi(\rho_{(a,i,i')}((b,j,j')))=\nu_{(\varphi(a),i,i')}((\varphi(b),j,j'))$. Hence $\psi$ is an epimorphism. Finally, for $(a,i,i')\in G\times \Z_n^2$
\[\Phi\circ \psi((a,i,i'))=\Phi((\varphi(a),i,i'))
=U_X^{-i}T_X^{-i}(\varphi(a)(e))=\Psi((a,i,i')). \qedhere
\]
\end{proof}

As a consequence, we may say that a solution can be obtained by the construction from Section~5 if and only if it is isomorphic to its canonical solution. Nevertheless, there is a simpler description.

\begin{proposition}
A finite indecomposable multipermutation solution ~$(X,\sigma,\tau)$ of level at most~$2$ is isomorphic to the solution of the form $\mathcal{S}(G\times \Z_n^2,\mathbf{c}, \mathbf{d})$, for some $n\in \mathbb{N}$, a group $(G,+,0)$ and sequences   $\mathbf{c}, \mathbf{d}\in G^n$, if and only if 
$\D{X}\cong\Z_n^2$,
$\LR{X}\cap \D{X}=\{\id\}$ and $\mathcal{GD}(X)_{\widetilde {0}}\subseteq\LR{X}$.
\end{proposition}
\begin{proof}
Let $(X,\sigma,\tau)$ be an indecomposable solution of level at most $2$. 

``$\Leftarrow$'' Let $\tilde 0\in X$, $\LR{X}\cap \D{X}=\{\id\}$ and $\mathcal{GD}(X)_{\widetilde {0}}\subseteq\LR{X}$. Let $(Y,\lambda,\rho)$ be the canonical solution of $(X,\sigma,\tau)$. 

By Corollary 
\ref{prop:canhom}, each indecomposable solution $(X,\sigma,\tau)$ of level at most~$2$ is a homomorphic image of its canonical solution. 
The mapping $\Phi\colon Y\to \dis{X}/{\dis{X}_{\widetilde {0}}}$ 
from the proof of Theorem~\ref{prop:hom-img} is given by
 $(\gamma{\dis{X}_{\widetilde {0}}},i,j)\mapsto U_X^{-i}T_X^{-j}\gamma(\widetilde {0}){\dis{X}_{\widetilde {0}}}$. Hence, it is sufficient to show that the mapping $\Phi$ is an injection. 
 
 Let $(\gamma,i,j),(\gamma',i',j')\in \dis{X}\times \Z_n^2$. Then 
 \[
U_X^{-i}T_X^{-j}\gamma(\widetilde {0})= U_X^{-i'}T_X^{-j'}\gamma'(\widetilde {0})\;\Leftrightarrow\;\gamma^{-1}U_X^{i-i'}T_X^{j-j'}\gamma'(\widetilde {0})= \widetilde {0}\;\Leftrightarrow\;\gamma^{-1}U_X^{i-i'}T_X^{j-j'}\gamma'\in \mathcal{GD}(X)_{\widetilde {0}}.
 \]
Assumption $\mathcal{GD}(X)_{\widetilde {0}}\subseteq\LR{X}$ forces $U_X^{i-i'}T_X^{j-j'}$ to belong to $\LR{X}$. Furthermore, by $\LR{X}\cap \D{X}=\{\id\}$, we obtain $U_X^{i-i'}T_X^{j-j'}=\id$. Hence, $(\gamma')^{-1}\gamma\in \dis{X}_{\widetilde {0}}$, which proves that $\gamma\dis{X}_{\widetilde {0}}=\gamma'\dis{X}_{\widetilde {0}}$. 
Finally, $\D{X}\cong\Z_n^2$ if and only if $\{U_X,T_X\}$ is a free basis of $\Z_n^2$ and hence $U_X^{i-i'}T_X^{j-j'}=\id$ if and only if $i=i'$ and $j=j'$.

``$\Rightarrow$'' It follows from 
Remark \ref{rem:cansol} and Lemma~\ref{lem:stabilizer}.
\end{proof}

We can also go in the opposite direction and construct the largest solution that maps homomorphically onto a given solution.

\begin{de}\label{def:universal}
Let $\bigoplus_\Z\Z$ be the free abelian group of rank~$\omega$ and let $\{e_i\mid i\in\Z\}$ and $\{f_i\mid i\in\Z\}$ be two disjoint subsets of~$\bigoplus_\Z\Z$ such that their union forms a free basis of~$\bigoplus_\Z\Z$. Let
\[
c_i=\begin{cases}
\sum_{k=1}^i e_k, & \text{for }i>0,\\
0, & \text{for }i=0,\\
\sum_{k=i+1}^{0} -e_{k}, & \text{for }i<0,
\end{cases}
\qquad\text{ and }\qquad
d_i=\begin{cases}
\sum_{k=1}^i f_k, & \text{for }i>0,\\
0, & \text{for }i=0,\\
\sum_{k=i+1}^{0} -f_{k}, & \text{for }i<0.
\end{cases}
\]
Then $\mathcal{S}((\bigoplus_\Z\Z)\times \Z^2,\mathbf{c},\mathbf{d})$, for ${\bf c}=(c_i)_{i\in \Z_n}$ and ${\bf d}=(d_i)_{i\in \Z_n}$, is called {\em the universal indecomposable solution of multipermutation level~$2$}.
\end{de}

We remark that, for $i\geq j$, we have $c_i-c_j=\sum_{k=j+1}^i e_j$ and $d_i-d_j=\sum_{k=j+1}^i f_j$ and hence 
$e_i=c_i-c_{i-1}$ and $f_i=d_i-d_{i-1}$, for all $i\in\Z$.
Considering that $c_0=d_0=0$, we obtain that
the free group $\bigoplus_\Z\Z$ is generated by $\{c_i,d_i\mid i\in\Z\}$ thus satisfying the conditions of Theorem~\ref{th:main2}.

\begin{proposition}\label{prop:img-universal}
    Each indecomposable solution of multipermutation level at most~$2$ is a homomorphic image of the universal indecomposable solution of multipermutation level~$2$.
\end{proposition}

\begin{proof}
    Let $(X, \sigma, \tau )$ be an indecomposable multipermutation solution of level at most~$2$ and choose an arbitrary element~$\tilde 0\in X$. To prove the result we use Proposition~\ref{prop:hom-img}. Note that Points (1) and (2) are assumed, Points (3) and (4) are trivial and hence we only have to check Point (5). Let $\varphi\colon \bigoplus_\Z\Z\to \dis{X}/\dis{X}_{\tilde 0}$ be such that:
    \[\varphi(e_i)=\sigma_{\tilde i}\sigma_{\widetilde{i-1}}^{-1}\dis{X}_{\tilde 0}
    \qquad\text{ and }\qquad
    \varphi(f_i)=\tau_{\tilde i}\tau_{\widetilde{i-1}}^{-1}\dis{X}_{\tilde 0}.
    \]
    Since $\{e_i:i\in\Z\}\cup\{f_i:i\in\Z\}$ is a free basis of $\bigoplus_\Z\Z$ we can uniquely extend $\varphi$ to the whole group.
    An easy computation then proves that
    the homomorphism~$\varphi$ satisfies (5).
\end{proof}

\section{Congruences}
It is well known that homomorphisms of solutions and equivalence relations preserving the structure of solutions are closely related. Let $(X,\sigma,\tau)$ be a solution. An equivalence relation $\mathord{\asymp}\subseteq X\times X$ such that for $x_1,x_2,y_1,y_2\in X$ 
\begin{align}\label{congr}
&x_1\asymp x_2\;\; {\rm and} \;\; y_1\asymp y_2\quad \Rightarrow\quad \sigma^{\varepsilon}_{x_1}(y_1)\asymp \sigma^{\varepsilon}_{x_2}(y_2)\quad {\rm and}\quad \tau^{\varepsilon}_{x_1}(y_1)\asymp \tau^{\varepsilon}_{x_2}(y_2),
\end{align}
where $\varepsilon\in \{-1,1\}$, is called a \emph{congruence} of the solution $(X,\sigma,\tau)$.

If $\Phi\colon Y\to X$ is a homomorphism
from a~solution~$(Y,\lambda,\rho)$ to a solution~$(X,\sigma,\tau)$ then the \emph{kernel} of $\Phi$, defined by 
\begin{align*}
&x_1\ker\Phi~x_2\quad \Leftrightarrow\quad \Phi(x_1)=\Phi(x_2)
\end{align*}
is a congruence of~$(Y,\lambda,\rho)$. 
On the other hand, for a congruence $\asymp$ of the solution $(Y,\lambda,\rho)$ 
let $\chi\colon Y\to Y/\mathord{\asymp}$ defined by $y\mapsto y_{\mathord{\asymp}}$ be the \emph{natural projection} from the solution to the quotient one. Clearly, $\chi$ is onto and $\ker\chi=\mathord{\asymp}$.

Moreover, for any epimorphism $\Phi\colon Y\to X$, the solution~$(X,\sigma,\tau)$ is isomorphic to the quotient solution $(Y,\lambda,\rho)/{\ker\Phi}$. Hence, a solution is a homomorphic image of a solution $(Y,\lambda,\rho)$ if and only if it is isomorphic to a quotient of  $(Y,\lambda,\rho)$ by some congruence. Thus the problem of finding all homomorphic images of  $\mathcal{S}(G\times\Z_n^2,\mathbf{c},\mathbf{d})$ reduces to the problem of finding all congruences of $\mathcal{S}(G\times\Z_n^2,\mathbf{c},\mathbf{d})$.
Since every epimorphism of solutions induces an epimorphism of the associated permutation groups, instead of looking directly at the solution homomorphism, we shall be analyzing the associated group homomorphism.
A group homomorphism sends subgroups onto subgroups and therefore, since the group $\GD{Y}$ is a semidirect product, we need to understand how to describe subgroups of a semidirect product. For this we use a classical result by Kurt Rosenbaum.
\begin{theorem}\cite{R88}\label{th:Rosenbaum}
	A nonempty set~$W$ of elements of a semidirect product of groups  $G=K\ltimes N$ is a subgroup of~$G$ if and only if
	\begin{enumerate}
		\item the sets $K_0=WN\cap K$ and $K_1=W\cap K$ are subgroups of~$G$,
		\item the set $N_1=W\cap N$ is a subgroup of $G$ and the set $N_0=WK\cap N$ is a collection of (not necessarily all) left $N_1$-cosets in~$N$,
		\item there is a 1-1 mapping $\eta$ from the set of all right $K_1$-cosets in~$K_0$ into the set
		of left $N_1$-cosets in~$N_0$
		that maps $K_1g$ onto the coset $nN_1$, with $(g,n)\in W$, satisfying $\eta(g_1g_2K_1)=g_2^{-1}\eta(g_1K_1)g_2\eta(g_2K_1)$.	
	\end{enumerate}
\end{theorem}

\begin{theorem}\label{thm:import}
Let $(Y,\lambda,\rho)=\mathcal{S}(G\times \Z_n^2,\mathbf{c}, \mathbf{d})$ be a solution for some $n\in \mathbb{N}$, a group $(G,+,0)$ and sequences   $\mathbf{c}, \mathbf{d}\in G^n$ and let $\asymp$ be a congruence of $(Y,\lambda,\rho)$. 
Then there exist
\begin{enumerate}
\item  a subgroup~$H$ of~$(G,+,0)$, 
\item a~subgroup~$S$ of~$(\Z_n^2,+)$,
\item a homomorphism $\Theta:S\to G/H$ such that
$c_{i-i'+k}+H=c_k+ H$ and $d_{i-i'+k}+H=d_k+H$, for any $k\in\Z_n$ whenever $(i,i')\in S$, and
\[(a,i,i')\asymp(b,j,j')\Longleftrightarrow (i-j,i'-j')\in S \text{ and } (a-b)+H=\Theta((i-j,i'-j')).\]
\end{enumerate}
On the other hand, a triple: subgroups $S\leq\Z_n^2$ and $H\leq G$ and a homomorphism $\Theta\colon S\to G/H$ 
described by $(1)-(3)$ defines a congruence of $\mathcal{S}(G\times\Z_n^2,\mathbf{c},\mathbf{d})$.
\end{theorem}
  \begin{proof}
Let $\asymp$ be a congruence of $(Y,\lambda,\rho)$ and $(X,\sigma,\tau):=(Y,\lambda,\rho)/_\asymp$. 
Then $(X,\sigma,\tau)$ is a homomorphic image of $(Y,\lambda,\rho)$ along with the natural projection $\Phi$. Let $e:=\Phi((0,0,0))=(0
,0,0)_{\asymp}$.
According to Proposition~\ref{prop:hom-img}, there exists
a homomorphism $\varphi:G\to\mathcal{LR}(X)/\LR{X}_{e}$ and
$\Phi((a,i,i'))=U_X^{-i}T_X^{-i'}(\varphi(a)(e))$.
By Lemma \ref{lm:GisoLR}, $(G,+,0)\cong \mathcal{LR}(Y)/\mathcal{LR}(Y)_{(0,0,0)}$ by some isomorphism $\iota$. Without loss of generality, we can assume $(G,+,0)=\mathcal{LR}(Y)/\mathcal{LR}(Y)_{(0,0,0)}$. 



We first describe the congruence class where the element $(0,0,0)$ belongs to. Since $e=\Phi((0,0,0))$, we have $(a,i,i')\asymp (0,0,0)$ if and only if $\Phi((a,i,i'))=e$
if and only if $U_X^{-i}T_X^{-i'}\varphi(a(e))=e$
if and only if $U_X^{-i}T_X^{-i'}\varphi(a)\subseteq \GD{X}_{e}$. 
Just now the question: ``How to describe the subgroup $\GD{X}_{e}$?'' arises. 

Obviously
$\GD{X}_{e}\cap \LR{X}$ is a subgroup of $\GD{X}$. Clearly $\GD{X}_{e}\cap \LR{X}=\LR{X}_e$. 
The other intersection, namely $\D{X}\cap \GD{X}_e$ is
also a subgroup of $\GD{X}$, namely $\D{X}\cap \GD{X}_e=\D{X}_e$.

Now we are interested in the preimage of $\GD{X}_e$ under the homomorphism $\Phi_{\mathcal{GD}}$. It is easy to see that it is a subgroup
of the semidirect product $\GD{Y}=\D{Y}\ltimes \LR{Y}$. Using the notation from Theorem~\ref{th:Rosenbaum}, let us denote: 
\begin{itemize}
\item $W:=\Phi_{\mathcal{GD}}^{-1}(\GD{X}_{e})$ 
\item $N:=\LR{Y}$ 
\item $K:=\D{Y}$ 
\item $K_1:=W\cap K=\Phi_{\mathcal{GD}}^{-1}(\GD{X}_{e})\cap\D{Y}=\{\alpha\in \D{Y}:\Phi_{\mathcal{D}}(\alpha)\in \D{X}_e\}$ 
\item $K_0=:WN\cap K=\{\alpha\in\D{Y}: \exists(\beta\in\LR{Y})\;\Phi_{\mathcal{GD}}(\alpha\beta)\in \GD{X}_e\}$ 
\item $N_1=W\cap N=\Phi_{\mathcal{GD}}^{-1}(\GD{X}_{e})\cap\LR{Y}=\{\alpha\in \LR{Y}:\Phi_{\mathcal{LR}}(\alpha)\in \LR{X}_e\}$
\item $N_0:=WK\cap N=\{\alpha\in\LR{Y}: \exists(\beta\in\D{Y})\;\Phi_{\mathcal{GD}}(\alpha\beta)\in \GD{X}_e\}$
\end{itemize}
Now
\begin{multline*}
N_1=\Phi_{\mathcal{GD}}^{-1}(\GD{X}_{e})\cap \LR{Y}
=\Phi_{\mathcal{GD}}^{-1}(\GD{X}_{e}\cap\LR{X})\cap \LR{Y}\\
=\Phi_{\mathcal{GD}}^{-1}(\LR{X}_e)\cap \LR{Y}=\Phi_{\mathcal{LR}}^{-1}(\LR{X}_e).
\end{multline*}

Next we prove that $[g,\alpha]\in N_1$, for each $\alpha \in \LR{Y}$ and $g\in K_0$. First note that, for $g\in K_0$, we have $g\in \D{Y}$ and there is $\beta\in \LR{Y}$ such that $\Phi_{\mathcal{GD}}(g\beta)\in \GD{X}_e$. Since $\Phi_{\mathcal{GD}}$ is a homomorphism we obtain $\Phi_{\mathcal{GD}}(g)\Phi_{\mathcal{GD}}(\beta)\in \GD{X}_e$. By Lemma \ref{lm:center}, $\Phi_{\mathcal{GD}}(g)\in Z(\GD{X})$. This gives that $\Phi_{\mathcal{GD}}([g,\alpha])=[\Phi_{\mathcal{GD}}(g),\Phi_{\mathcal{GD}}(\alpha)]=\id_X$. Hence, $\Phi_{\mathcal{GD}}([g,\alpha])\in \GD{X}_e$, and $[g,\alpha]\in W=\Phi_{\mathcal{GD}}^{-1}(\GD{X}_{e})$. Further, since the group $N=\LR{Y}$ is a normal subgroup of $\GD{Y}$ and $\alpha\in \LR{Y}$, we obtain $[g,\alpha]\in N$, and in  consequence, $[g,\alpha]\in W\cap N=N_1$.

Let $\eta\colon K_0/ K_1\to N_0/N_1$ be the mapping from Theorem~\ref{th:Rosenbaum}, where $\eta(gK_1):=nN_1$, with $g\in K_0$, $n\in N_0$ and $gn\in W$. Let $\eta(g_1K_1):=n_1N_1$ and $\eta(g_2K_1):=n_2N_1$, for some $g_1,g_2\in K_0$ and $n_1, n_2\in N_0$. Since $[g_2,n_1]\in N_1$ 
and $[g_2,N_1]\subseteq N_1$, 
it turns out that $\eta$ is a homomorphism:
\begin{multline*}
\eta(g_1g_2K_1)=g_2^{-1}\eta(g_1K_1)g_2\eta(g_1K_1)=g_2^{-1}n_1N_1g_2n_2N_1=g_2^{-1}n_1g_2n_2N_1\\=g_2^{-1}n_1g_2n_1^{-1}n_1n_2N_1=n_1n_2N_1=\eta(g_1K_1)\eta(g_2K_1),
\end{multline*}
for all $g_1,g_2\in K_0$ and $n_1,n_2\in N_0$.

The homomorphism $\eta$ from $K_0/ K_1$ to $N/N_1$ can be naturally extented into a homomorphism $\overline{\eta}$ from the group $K_0$ to the quotient $N/N_1$. By Remark \ref{rem:cansol}, there is an isomorphism $\xi\colon \Z_n^2\to \D{Y}$ with $\xi((i,j))= U_Y^{-i}T_Y^{-j}$. Moreover, since 
$G= N/N_{(0,0,0)}$ and $N_1\geq N_{(0,0,0)}$,
we can take~$S=\xi^{-1}(K_0)$, $H=N_1/N_{(0,0,0)}$ and $\Theta=\overline{\eta}\xi|_S$. By the 3rd Isomorphism Theorem, $\Theta\colon S\to N/N_1\cong G/H$.

Let us now take two elements $(a,i,i')$ and $(b,j,j')$ from $G\times\Z_n^2$ 
such that $a((0,0,0))=(a,0,0)$ and $b((0,0,0))=(b,0,0)$. Then 
\begin{multline*}
	(a,i,i')\asymp (b,j,j') \Leftrightarrow
	\Phi((a,i,i'))=\Phi((b,j,j'))
	\Leftrightarrow\\
	\Phi(U_Y^{-i}T_Y^{-i'}a((0,0,0)))=
	\Phi(U_Y^{-j}T_Y^{-j'}b((0,0,0)))\\
	\Leftrightarrow
	\Phi_{\mathcal{GD}}(U_Y^{-i}T_Y^{-i'}a)\Phi((0,0,0))=
    \Phi_{\mathcal{GD}}(U_Y^{-j}T_Y^{-i'}b)\Phi((0,0,0))\\
	\Leftrightarrow
    U_X^{-i}T_X^{-i'}\Phi_{\mathcal{GD}}(a)\Phi((0,0,0))=
   U_X^{-j}T_X^{-j'}\Phi_{\mathcal{GD}}(b)\Phi((0,0,0))\\
	\Leftrightarrow
    (\Phi_{\mathcal{GD}}(b))^{-1}U_X^{j-i}T_X^{j'-i'}\Phi_{\mathcal{GD}}(a)\Phi((0,0,0))=\Phi((0,0,0))\\
	\Leftrightarrow
    (\Phi_{\mathcal{GD}}(b))^{-1}U_X^{j-i}T_X^{j'-i'}\Phi_{\mathcal{GD}}(a)\in \GD{X}_e\\
	\Leftrightarrow
    \Phi_{\mathcal{GD}}(b^{-1}U_Y^{j-i}T_Y^{j'-i'}a)\in \GD{X}_e\
	\Leftrightarrow\
    b^{-1}U_Y^{j-i}T_Y^{j'-i'}a\in \Phi_{\mathcal{GD}}^{-1}(\GD{X}_e)=W.
    \end{multline*}
This implies that $U^{j-i}T^{j'-i'}\in K_0$ and hence $[b,U^{i-j}T^{i'-j'}]\in N_1$.
Note that
\[
b^{-1}U_Y^{j-i}T_Y^{j'-i'}a=b^{-1}U_Y^{j-i}T_Y^{j'-i'}b U_Y^{i-j}T_Y^{i'-j'}U_Y^{j-i}T_Y^{j'-i'}b^{-1}a\quad \Rightarrow\quad U_Y^{j-i}T_Y^{j'-i'}b^{-1}a\in W,
\]
since 
\[b^{-1}U_Y^{j-i}T_Y^{j'-i'}b U_Y^{i-j}T_Y^{i'-j'}U_Y^{j-i}T_Y^{j'-i'}b^{-1}a=[b,U_Y^{i-j}T_Y^{i'-j'}]U_Y^{j-i}T_Y^{j'-i'}b^{-1}a
.
\]
In particular, \begin{align*}
	(a,0,0)\asymp (0,0,0)\quad \Leftrightarrow\quad a\in W\quad\Rightarrow \quad a\in N_1.
    \end{align*}
Further,
    \begin{multline*}
	b^{-1}U_Y^{j-i}T_Y^{j'-i'}a\in \Phi_{\mathcal{GD}}^{-1}(\GD{X}_e)=W \quad  \Leftrightarrow\quad 
    	U_Y^{j-i}T_Y^{j'-i'}\in K_0 \land U_Y^{j-i}T_Y^{j'-i'}b^{-1}a\in W 
	\Leftrightarrow\\ 
	U_Y^{j-i}T_Y^{j'-i'}\in K_0 \land \overline{\eta}(U_Y^{j-i}T^{j'-i'}) =b^{-1}a N_1
	\Leftrightarrow
	(j-i,j'-i')\in S \land \Theta((j-i,j'-i'))=b^{-1}aH.
\end{multline*}

Finally, if $(i,i')\in S$ then $U_Y^{-i}T_Y^{-i'}\in K_0$ and there is $\alpha\in N$ such that $U_Y^{-i}T_Y^{-i'}\alpha\in W$. Hence by Lemma 
\ref{lem:alpha} we have, for each $j\in\Z_n$:
\begin{multline*}
\Phi((c_{i-i'+j},0,0))=\Phi(\textit{Ł}_{(0,j,0)}\textit{Ł}_{(0,i,i')}((0,0,0)))
=\Phi_{\mathcal{GD}}(\textit{Ł}_{(0,j,0)}\textit{Ł}_{U_Y^{-i}T_Y^{-i'}((0,0,0))})\Phi((0,0,0))\\
=\Phi_{\mathcal{GD}}(\textit{Ł}_{(0,j,0)}\textit{Ł}_{U_Y^{-i}T_Y^{-i'}\alpha((0,0,0))})\Phi((0,0,0))
=L_{\Phi((0,j,0))}L_{\Phi_{\mathcal{GD}}(U_Y^{-i}T_Y^{-i'}\alpha)\Phi((0,0,0)))})\Phi((0,0,0))\\
=L_{\Phi((0,j,0))}L_{\Phi_{\mathcal{GD}}(U_Y^{-i}T_Y^{-i'}a)(e)}(e)=L_{\Phi((0,j,0))}L_e(e)=L_{\Phi((0,j,0))}(e)=L_{\Phi((0,j,0))}\Phi((0,0,0))\\
=\Phi(\textit{Ł}_{(0,j,0)}((0,0,0)))=\Phi((c_j,0,0)).
\end{multline*}
Hence $(c_{i-i'+j},0,0)\asymp (c_j,0,0)$ and therefore
$c_{i-i'+j}-c_j\in \Theta((0,0))=H$. The argument is the same to show that $d_{i-i'+j}-d_j\in H$.
This finishes the proof of the forward implication. 
\vskip 2mm

For the backward implication let $(a,i,i'), (b,j,j'), (g,k,k'), (h,l,l')\in G\times \Z_n^2$ and suppose $(a,i,i')\asymp(b,j,j')$ and $(g,k,k')\asymp(h,l,l')$ where $\asymp$ is defined in~(3). Then $(i-j,i'-j'), (k-l,k'-l')\in S$, $(g-h)+H=\Theta((k-l,k'-l'))$ and $c_{i-j-i'+j'+A}+H=c_A+H$ and $c_{k-l+l'-k'+B}+H=c_B+H$, for all $A,B\in \Z_n$. We prove that $\asymp$ is a congruence of $(Y,\lambda,\rho)$. Clearly we have
\begin{align*}
    \lambda_{(a,i,i')}((g,k,k'))&=(g+c_{i-i'+k'-k-1},k+1,k'),\\
    \lambda_{(b,j,j')}((h,l,l'))&=(h+c_{j-j'+l'-l-1},l+1,l').
\end{align*}
These two elements are equivalent modulo $\asymp$ since $(k+1-(l+1),k'-l')\in S$ and taking $A=B=k'-k+j-j'-1$ we obtain: 
$c_{i-i'+k'-k-1}+H=c_{j-j'+k'-k-1}+H=c_{j-j'+l'-l-1}
+H$. Hence $g+c_{i-i'+k'-k-1}-(h+c_{j-j'+l'-l-1})+H=g-h+H=\Theta((k-l,k'-l'))$.
Similarly we can repeat the consideration for the operation $\rho$, which completes the proof.
\end{proof}

Let $\asymp$ be a congruence of a solution $(Y,\lambda,\rho)=\mathcal{S}(G\times \Z_n^2,\mathbf{c}, \mathbf{d})$ determined by subgroups $S\leq\Z_n^2$ and $H\leq G$ and a homomorphism $\Theta\colon S\to G/H$. Assume that $(X,\sigma,\tau):=(Y,\lambda,\rho)/_\asymp$. Then $(X,\sigma,\tau)$ is a homomorphic image of $(Y,\lambda,\rho)$ along with the natural projection $\Phi\colon Y\to Y/\mathord{\asymp}$.
\begin{lemma}\label{lem:7.3}
 Let $e:=\Phi((0,0,0))=(0,0,0)/\mathord{\asymp}$. Then:
 \begin{itemize}
 \item   $\dis{X}/\dis{X}_e\cong \mathcal{LR}(X)/\mathcal{LR}(X)_e\cong G/H$,
 \item $\mathcal{D}(X)\cong \Z_n^2/\Ker\Theta$,
 \item $\mathcal{GD}(X)/(\mathcal{GD}(X)_e\mathcal{LR}(X))\cong \Z_n^2/S$.
 \end{itemize}
 In particular, $(i,j)\in S$ if and only if there is $\alpha\in \LR{X}$ such that $U^iT^j\alpha\in (\mathcal{GD}(X)_e$.
\end{lemma}

\begin{proof}
The induced homomorphism $\Phi_{\mathcal{GD}}\colon \GD{Y}\to \GD{X}$ is onto and clearly $\mathcal{LR}(X)=\Phi_{\mathcal{GD}}(N)$, $\mathcal{LR}(X)_e=\Phi_{\mathcal{GD}}(N_1)$,
$\mathcal{D}(X)=\Phi_{\mathcal{GD}}(K)$, $\mathcal{D}(X)_e=\Phi_{\mathcal{GD}}(K_1)$ and $\mathcal{GD}(X)_e\mathcal{LR}(X)=\Phi_{\mathcal{GD}}(WN)$, for subgroups $N$, $N_1$, $K$ and $K_1$ defined in the proof of Theorem \ref{thm:import}.

According to the 2nd Isomorphism Theorem $\mathcal{GD}(Y)/WN =KN/WN\cong K/WN\cap K=K/K_0$.
Further, by 
Proposition~\ref{prop:aut-regular}, the automorphism group of the solution $(X,\sigma,\tau)$ is regular and hence $\D{X}_e$ is trivial. Moreover, $(i,j)\in \Ker\Theta$ if and only if $U_Y^{-i}T_Y^{-j}\in K_1$, which finishes the proof.
\end{proof}

\begin{proposition}\label{prop:order}
    $|X|=[\Z_n^2:S]\cdot[G:H]$.
\end{proposition}

\begin{proof}
By Lemma \ref{lem:7.3} and the 2nd Isomorphism Theorem we have:
    \begin{multline*}
        |X| = [\mathcal{GD}(X):\mathcal{GD}(X)_e]
        = [\mathcal{GD}(X):\mathcal{GD}(X)_e\mathcal{LR}(X)]\cdot
        [\mathcal{GD}(X)_e\mathcal{LR}(X):\mathcal{GD}(X)_e]=\\
             [\Z_n^2:S]\cdot
        [\mathcal{LR}(X):\mathcal{GD}(X)_e\cap \mathcal{LR}(X)]=
        [\Z_n^2:S]\cdot
        [\mathcal{LR}(X): \mathcal{LR}(X)_e]=
        [\Z_n^2:S]\cdot[G:H] \qedhere
    \end{multline*}
\end{proof}

\begin{corollary}\label{cor:order_D}
    $|\D{X}|$ divides $|X|$.
\end{corollary}

\begin{proof}
    $|X|=|\Z_n^2/S|\cdot|G/H|=
    |\Z_n^2/S|\cdot|\mathop{\mathrm{Im}}\Theta|\cdot|G/H:\mathop{\mathrm{Im}}\Theta|=
    |\Z_n^2/S|\cdot|S/\Ker\Theta|\cdot|G/H:\mathop{\mathrm{Im}}\Theta|$
    $\strut=|\D{X}|\cdot|G/H:\mathop{\mathrm{Im}}\Theta|$.
\end{proof}

\begin{proposition}
    There exist $k_\sigma,k_\tau\in\mathbb{N}\cup\{\infty\}$ such that each $\sigma_x$, for $x\in X$, is a permutation consisting of cycles of length~$k_\sigma$ and each $\tau_x$ is a permutation consisting of cycles of length~$k_\tau$.
\end{proposition}

\begin{proof}
    Let $x,e\in X$ be taken arbitrary and let $k$ be the least positive integer such that $\sigma_x^k(e)=e$ if such a number exists.
    Suppose $\Phi((0,0,0))=e$ and $\Phi((a,i,j))=x$ for some $(a,i,j)\in G\times \Z_n^2$. Then 
    \begin{align*}
\Phi((\sum_{\ell=1}^k c_{i-j-\ell},k,0))=\Phi(\lambda^k_{(a,i,j)}((0,0,0)))=\Phi_{\mathcal{GD}}
(\lambda^k_{(a,i,j)})\Phi((0,0,0))=\sigma^k_x\Phi((0,0,0))=\Phi((0,0,0))
    \end{align*} 
    and $(0,0,0)\asymp (\sum_{\ell=1}^k c_{i-j-\ell},k,0)$.
    This means $(k,0)\in S$ and $\Theta((k,0))=\sum_{\ell=1}^k c_{i-j-\ell}+H$.

    We now prove, by an induction on~$m$, that
    $\Theta((k,0))=\sum_{\ell=1}^k c_{m+i-j-\ell}+H$,
    for any $m\in\Z$. The claim is trivial for~$m=0$. Let us now suppose that it holds for
    some $m\geq 0$.
    Then
    \begin{multline*}
    \sum_{\ell=1}^k c_{m+1+i-j-\ell}+H
    =c_{m+i-j}+\sum_{\ell=2}^{k} c_{m+i-j-(\ell-1)}+H=  \\
     =c_{m+i-j}+\sum_{\ell=1}^{k-1} c_{m+i-j-\ell}+c_{m-k+i-j}-c_{m-k+i-j}+H=  \\
     =c_{m+i-j}+\sum_{\ell=1}^{k} c_{m+i-j-\ell}-c_{m-k+i-j}+H=\Theta((k,0))
         \end{multline*}
         since  $c_{-k+m+i-j}-c_{m+i-j}\in H$
         because $(k,0)\in S$.
    Therefore the sum of any $k$ consecutive $c_i$ is the image of~$\Theta(k,0)$, modulo~$H$.
    
    Let now $(b,r,s), (b',r',s')\in G\times \Z_n^2$ be arbitrary elements.
    Then $\lambda_{(b,r',s')}^k((b,r,s))=(b+\sum_{\ell=1}^k c_{r'-r-s'+s-\ell},r+k,s)$ and
    $(r+k,s)-(r,s)\in S$. Further, 
    \[
    b+\sum_{\ell=1}^k c_{r'-r-s'+s-\ell}-b+H=
    \sum_{\ell=1}^k c_{r'-r-s'+s-\ell}+H=\Theta((k,0))=\Theta((r+k,s)-(r,s)).
    \]
    This means that $(b,r,s)\asymp(b+\sum_{\ell=1}^k c_{r'-r-s'+s-\ell},r+k,s)$ and hence $(\Phi_{\mathcal{GD}}(\lambda_{(b',r',s')}))^k \Phi((b,r,s))=\Phi((b,r,s))$. The proof for $\tau$ is analogous.
\end{proof}

\begin{proposition}
    A solution $(X,\sigma,\tau)$ is involutive if and only if $(1,1)\in S$, $\Theta((1,1))=H$ and $c_i+d_i\in H$, for each~$i\in\Z_n$.
\end{proposition}

\begin{proof}
 By Lemma \ref{lem:inv} a solution $(X,\sigma,\tau)$ is involutive if and only if $U_XT_X=\mathrm{id}$ and $\sigma_x U_X\tau_xT_X=\mathrm{id}$, for all~$x\in X$. At first note that for $(a,i,i')\in G\times \Z_n^2$, 
    $U_YT_Y((a,i,i'))=(a,i-1,i'-1)$ and $(a,i,i')\asymp(a,i-1,i'-1)$ if and only if $(1,1)\in S$ and $(a-a)+H=\Theta((1,1))$.
    Then for $(b,j,j')\in G\times \Z_n^2$, 
    \[\lambda_{(b,j,j')} U_Y\rho_{(b,j,j')}T_Y ((a,i,i')) = (a+c_{j-j'-i+i'}+d_{j-j'-i+i'},i,i') \]
    and $(a,i,i')\asymp (a+c_{j-j'-i+i'}+d_{j-j'-i+i'},i,i')$ if and only if $c_{j-j'-i+i'}+d_{j-j'-i+i'}\in H$.
\end{proof}

\begin{proposition}
    A solution $(X,\sigma,\tau)$ is of multipermutation level~$1$ if and only if $G=H$.
\end{proposition}

\begin{proof}
    Clearly, $\textit{Ł}_{(0,0,0)}((0,0,0))=(0,0,0)$ and $\textit{Ł}_{(0,i,0)}((0,0,0))=(c_i,0,0)$, for all $i\in\Z_n$. Moreover,  $(0,0,0)\asymp (c_i,0,0)$ only if $c_i\in H$. The same argument can be used  to show that also $d_i\in H$.
    Since the group $(G,+,0)$ is generated by~all $c_i$ and $d_i$,
    we obtain $H=G$.

    On the other hand, if $H=G$, then $(a,i,j)\asymp(b,i,j)$ for any $a,b\in G$. Hence, $\textit{Ł}_{(0,0,0)}((a,i,j))=(a+c_{j-i},i,j)\asymp(a,i,j)$ which gives $L_{(0,0,0)/_\asymp}=\mathrm{id}_X=\mathbf{R}_{(0,0,0)/_\asymp}$ and therefore $\sigma_x=U_X^{-1}$ and $\tau_x=T_X^{-1}$, for all $x\in X$.
\end{proof}

\begin{proposition}
    Let $\asymp_1$ and $\asymp_2$ be two congruences of a solution $\mathcal{S}(G\times \Z_n^2,\mathbf{c}, \mathbf{d})$. Then the quotient solution $(G\times\Z_n^2,\mathbf{c},\mathbf{d})/\asymp_1$ is isomorphic to $(G\times\Z_n^2,\mathbf{c},\mathbf{d})/\asymp_2$ if and only if $\mathop{\asymp_1}=\mathop{\asymp}_2$.
\end{proposition}

\begin{proof}
For $i=1,2$ let congruences $\asymp_i$ be determined by subgroups $S_i\leq\Z_n^2$, $H_i\leq G$ and group homomorphisms $\Theta_i:S_i\to G/H_i$.
    Let $\Phi: X_1\to X_2$ be an isomorphism of the solutions $(G\times\Z_n^2,\mathbf{c},\mathbf{d})/\asymp_1$ and $(G\times\Z_n^2,\mathbf{c},\mathbf{d})/\asymp_2$. Without loss of generality we can assume
    $\Phi((0,0,0)/_{\asymp_1})=(0,0,0)/_{\asymp_2}$
    since the automorphism groups of the solutions are regular.
    By Lemma \ref{lem:7.3}, for each $(i,j)\in \Z_n^2$, we have
    \begin{multline*}
        (i,j)\in S_1 \quad  \Rightarrow \quad
        \exists \alpha\in\LR{X_1}\quad U_{X_1}^iT_{X_1}^j\alpha((0,0,0)/_{\asymp_1})=(0,0,0)/_{\asymp_1}\\
        \Rightarrow\quad
        \Phi(U_{X_1}^iT_{X_1}^j\alpha((0,0,0)/_{\asymp_1}))=\Phi((0,0,0)/_{\asymp_1})\\
        \Rightarrow\quad
        U_{X_2}^iT_{X_2}^j\Phi_{\mathcal{LR}}(\alpha)((0,0,0)/_{\asymp_2})=(0,0,0)/_{\asymp_2}
        \quad \Rightarrow (i,j)\in S_2.
    \end{multline*}
    Let now $g\in G$. Then there exist $r_0,\ldots,r_{n-1},s_0,\ldots,s_{n-1}\in\Z_n$ such that $g=\sum_{k=0}^{n-1} r_kc_k+\sum_{k=0}^{n-1} s_kd_k$. 
    Now 
    \begin{multline*}
        g\in H_1 \quad\Rightarrow\quad
        (0,0,0)\asymp_1(g,0,0)=\prod_{k=0}^{n-1} \textit{Ł}_{1,(0,k,0)}^{r_k} \prod_{k=0}^{n-1} \slR_{1,(0,k,0)}^{s_k} ((0,0,0))\\
        \Rightarrow\quad
        (0,0,0)/_{\asymp_1}= \prod_{k=0}^{n-1} \left({T_{X_1}^k}L_{1,(0,0,0)/_{\asymp_1}} T_{X_1}^{-k}\right)^{r_k} \prod_{k=0}^{n-1} \left({T_{X_1}^k}\mathbf{R}_{1,(0,0,0)/_{\asymp_1}} T_{X_1}^{-k}\right)^{s_k} ((0,0,0)/_{\asymp_1})\\
        \Rightarrow\quad
        \Phi((0,0,0)/_{\asymp_1})=\Phi\left(\prod_{k=0}^{n-1} \left({T_{X_1}^k}L_{1,(0,0,0)/_{\asymp_1}} T_{X_1}^{-k}\right)^{r_k} \prod_{k=0}^{n-1} \left({T_{X_1}^k}\mathbf{R}_{1,(0,0,0)/_{\asymp_1}} T_{X_1}^{-k}\right)^{s_k} ((0,0,0)/_{\asymp_1})\right)\\
        \Rightarrow\quad
        (0,0,0)/_{\asymp_2}=\prod_{k=0}^{n-1} \left({T_{X_2}^k}L_{2,(0,0,0)/_{\asymp_2}} T_{X_2}^{-k}\right)^{r_k} \prod_{k=0}^{n-1} \left({T_{X_2}^k}\mathbf{R}_{2,(0,0,0)/_{\asymp_2}} T_{X_2}^{-k}\right)^{s_k} ((0,0,0)/_{\asymp_2})\\
        \Rightarrow \quad (0,0,0)\asymp_2\prod_{k=0}^{n-1} \textit{Ł}_{2,(0,k,0)}^{r_k} \prod_{k=0}^{n-1} \slR_{2,(0,k,0)}^{s_k} ((0,0,0))=(g,0,0)\;\Rightarrow\; (0,0,0)\asymp_2 (g,0,0)
        \;\Rightarrow\; g\in H_2.
    \end{multline*}
    Hence $S_1=S_2$ and $H_1=H_2$.

    Finally, let $(i,j)\in S_1$ and $g\in G$ be presented as before. Then
    \begin{multline*}
        g+H_1= \Theta_1((i,j)) \quad\Rightarrow\quad
        (0,0,0) \asymp_1(g,i,j)\\
        \Rightarrow \quad
        (0,0,0)\asymp_1 U_{X_1}^{-i}T_{X_1}^{-j}
        \prod_{k=0}^{n-1} \textit{Ł}_{1,(0,k,0)}^{r_k} \prod_{k=0}^{n-1} \slR_{1,(0,k,0)}^{s_k} ((0,0,0))\\
        \Rightarrow \quad
        (0,0,0)/_{\asymp_1}= U_{X_1}^{-i}T_{X_1}^{-j}
        \prod_{k=0}^{n-1} \left({T_{X_1}^k}L_{1,(0,0,0)/_{\asymp_1}} T_{X_1}^{-k}\right)^{r_k}
        \prod_{k=0}^{n-1} \left({T_{X_1}^k}\mathbf{R}_{1,(0,0,0)/_{\asymp_1}} T_{X_1}^{-k}\right)^{s_k} ((0,0,0)/_{\asymp_1})\\
        \Rightarrow
        \Phi((0,0,0)/_{\asymp_1})= \Phi\left(U_{X_1}^{-i}T_{X_1}^{-j}
        \prod_{k=0}^{n-1} \left({T_{X_1}^k}L_{1,(0,0,0)/_{\asymp_1}}
        T_{X_1^{-k}}\right)^{r_k}
        \prod_{k=0}^{n-1} \left({T_{X_1}^k}\mathbf{R}_{1,(0,0,0)/_{\asymp_1}} T_{X_1}^{-k}\right)^{s_k} ((0,0,0)/_{\asymp_1})\right)\\
        \Rightarrow \quad
        (0,0,0)/_{\asymp_2}= U_{X_2}^{-i}T_{X_2}^{-j}
        \prod_{k=0}^{n-1} \left({T_{X_2}^k}L_{2,(0,0,0)/_{\asymp_2}} T_{X_2}^{-k}\right)^{r_k}
        \prod_{k=0}^{n-1} \left({T_{X_2}^k}\mathbf{R}_{2,(0,0,0)/_{\asymp_2}} T_{X_2}^{-k}\right)^{s_k} ((0,0,0)/_{\asymp_2})\\
        \Rightarrow \quad
        (0,0,0)\asymp_2 U_{X_2}^{-i}T_{X_2}^{-j}
        \prod_{k=0}^{n-1} \textit{Ł}_{2,(0,k,0)}^{r_k} \prod_{k=0}^{n-1} \slR_{2,(0,k,0)}^{s_k} ((0,0,0)) \\
        \Rightarrow\quad (0,0,0)\asymp_2 (g,i,j)
        \quad\Rightarrow \quad g+H_1=\Theta_2((i,j)).
    \end{multline*}
    Hence $\Theta_1=\Theta_2$. The other implication is trivial.
\end{proof}


We have now a tool how to find (and possibly enumerate) all indecomposable solutions of multipermutation level~$2$. Suppose that a solution $(X,\sigma,\tau)$ has a finite size~$|X|$. We know that it can be obtained as a homomorphic image of the universal solution. But there exists a much smaller solution from which we can construct our solution~$(X,\sigma,\tau)$, even without knowledge of the permutation groups acting on~$X$.

\begin{proposition}\label{prop:7.10}
Let $(X,\sigma,\tau)$ be a finite indecomposable solution of multipermutation level~$2$ of 
size~$n$.
Then $(X,\sigma,\tau)$ is a homomorphic image of~$\mathcal{S}(\Z_n^{2n-2}\times\Z_n^2,\mathbf{c},\mathbf{d})$ with the series 
${\bf c}=(c_i)_{i\in\Z_n}$ and ${\bf d}=(d_i)_{i\in\Z_n}$ defined analogously as in Definition~\ref{def:universal}. 
\end{proposition}

\begin{proof}
    Let $(Y,\lambda,\rho)$ be the universal solution
    and let $\asymp$ be a congruence such that
    $(Y,\lambda,\rho)/_{\asymp}\cong(X,\sigma,\tau)$.
    This congruence is determined by a certain triple: a subgroup  $S$ of $\Z^2$, a subgroup $H$ of the free abelian group $G=\bigoplus_\Z\Z$ and a homomorphism $\Theta\colon S\to G/H$.
    According to Proposition ~\ref{prop:order}, the orders of the groups $\Z^2/S$ and $G/H$ divide~$n$ and hence $n\cdot \Z^2\leq S$ and $n\cdot G\leq H$. We have also $n\cdot S\leq\Ker\Theta$
    and therefore $c_{i+n}-c_i\in H$ and $d_{i+n}-d_i\in H$, for any $i\in\Z$. Further, let $J=\langle c_1,\ldots,c_{n-1},d_1,\ldots,d_{n-1}\rangle$ and 
    $K=J\cap H$. Obviously, by the 2nd Isomorphism Theorem,  $J/K\cong JH/H=G/H$ and $n\cdot J\leq K$.

    Let now consider the solution $\mathcal{S}(\Z_n^{2n-2}\times\Z_n^2,\mathbf{c},\mathbf{d})$ and let us take $S'=S/(n\cdot \Z^2)$, $H'=K/(n\cdot J)$
    and $\Theta'((i,j))=\Theta((i\bmod n,j\bmod n))$.
    It is clear that the quotient solution $\mathcal{S}(\Z_n^{2n-2}\times\Z_n^2,\mathbf{c},\mathbf{d})/\propto$ by the congruence $\propto$ determined by subgroups~$S'$ of $\Z_n^2$ and $H'$ of $\Z_n^{2n-2}$
    and the homomorphism $\Theta'\colon S'\to \Z_n^{2n-2}/H'$, is isomorphic to~$(X,\sigma,\tau)$.
\end{proof}

Looking at the proof of Proposition \ref{prop:7.10}, 
we can actually obtain a stronger result. If we are looking for subgroups $S$ of $\Z_n^2$ and $H$ of $G$ to construct the solution $(Y,\lambda,\rho)$, we can take
$G=\Z_n^{2\cdot [\Z_n^2:S]-2}$ instead 
$G=\Z_n^{2\cdot |X|-2}$, which can simplify our considerations.


\begin{example}
    We shall construct all 19 indecomposable solutions of the multipermutation level~$2$ with having size 4. 
    They can be constructed as homomorphic images of $\mathcal{S}(\Z_4^{2\cdot [\Z_4^2:S]-2}\times\Z_4^2,\mathbf{c},\mathbf{d})$,
    by Proposition~\ref{prop:7.10}.
    Since, by Proposition \ref{prop:order},  $[\Z_4^2:S]\cdot[G:H]=4$, we have three possibilities:
    \begin{itemize}
        \item $[\Z_4^2:S]=4$ and $[G:H]=1$. Then
        $S$ is a subgroup of~$\Z_4^2$ of order 4 and there are 7 such subgroups: 6 cyclic groups $\langle(1,1)\rangle$, $\langle(1,2)\rangle$, $\langle(1,3)\rangle$, $\langle(1,0)\rangle$, $\langle(0,1)\rangle$,
        $\langle(2,1)\rangle$ and one elementary abelian group $2\cdot\Z_4^2$. There is
        only one choice of~$H$, namely $H=G$, and only one choice of~$\Theta$ for each choice of~$S$. Hence we have 7 solutions, all of them of multipermutation level~$1$.
        Only one of the subgroups $S$ contains $(1,1)$ hence only one of the solutions is involutive.
    \item $[\Z_4^2:S]=2$ and $[G:H]=2$. In this case, $G=\Z_4^2$. There are three subgroups of $\Z_4^2$ of size 8. Only one of them is eligible as a subgroup $S$ because $(1,0)\in S$ or $(0,1)\in S$ and both imply $H=G$. Hence let $S=\langle (1,1),(2,0)\rangle$. On the other side, for each subgroup $H$ of $G$,
    we have 4 choices of~$\Theta$, since each of the generators of~$S$ can be sent either onto~$H$ or onto $G\smallsetminus H$. Hence there are $1\cdot 3\cdot 4=12$ different solutions.
    How many of them are involutive?
    All of them satisfy $(1,1)\in S$ but only one of the choices of $H$ satisfies $c_1+d_1\in H$,
    namely $H=\langle(1,1)\rangle$. Moreover,
    only two of choices of $\Theta$ satisfy $\Theta((1,1))=H$ and therefore only two of the solutions are involutive.
    \item $[\Z_4^2:S]=1$ and $[G:H]=4$. This gives $|G|=1$ which contradicts the second condition.
    \end{itemize}
    \end{example}

\end{document}